\documentclass[dvipsnames, reqno, 10pt]{amsart}

\usepackage{amsfonts,amsthm,amssymb,amsmath,stackrel, latexsym,mathabx}

\usepackage{mathrsfs}
\usepackage{tikz-cd}

\usepackage{newcent}

\usepackage{xcolor}
\usepackage{graphicx}
\usepackage[all]{xy}

\usepackage{microtype}
\usepackage{xcolor}
\usepackage{graphicx}
\usepackage[all]{xy}
\usepackage[colorlinks=true,citecolor=MidnightBlue,linkcolor=MidnightBlue,urlcolor=MidnightBlue] %,backref=page]
           {hyperref}

           \usepackage[inline]{enumitem}
\usepackage{a4wide}
\usepackage{stmaryrd}
\usepackage{mathtools,
  thmtools}

\makeatletter
\DeclareRobustCommand*\cal{\@fontswitch\relax\mathcal}
\makeatother

\usepackage{newcent}
\usepackage{mathabx}
\usepackage{bbm}

\usepackage{pifont}

\usepackage{manfnt}

\usepackage{tikz-cd}
\usepackage{mathtools}

\usepackage{tikz}

\usepackage{stmaryrd}

\usepackage{capt-of}

\makeatletter
\@namedef{subjclassname@2020}{%
  \textup{2020} Mathematics Subject Classification}
\makeatother

\usepackage[OT2, T1]{fontenc} 

\usepackage[textsize=scriptsize]{todonotes}

\usepackage{scalerel,stackengine}  %%%%% for \pig(

\newcommand\pig[1]{\scalerel*[5pt]{\big#1}{%
  \ensurestackMath{\addstackgap[1.5pt]{\big#1}}}}

\newcommand{\compactlist}[1]{\setlength{\itemsep}{0pt} \setlength{\parskip}{0pt} \setlength{\leftskip}{-1.#1em}}

\newcommand{\Sum}{\textstyle\sum}

\numberwithin{equation}{section}

\DeclareRobustCommand{\SkipTocEntry}[5]{}

\theoremstyle{plain}

\newtheorem{theorem}{Theorem}[section]

\newtheorem{prop}[theorem]{Proposition}
\newtheorem{lemma}[theorem]{Lemma}
\newtheorem{lem}[theorem]{Lemma}

\newtheorem{cor}[theorem]{Corollary}

\theoremstyle{definition}
\newtheorem{definition}[theorem]{Definition}

\newtheorem{example}[theorem]{Example}

\newtheorem{rem}[theorem]{Remark}

\declaretheorem[name=Theorem, style=italics,
  numbered=no]{theorem*}

\newcommand{\ga}{\alpha}

\newcommand{\gD}{\Delta} 
\newcommand{\gve}{\varepsilon} 
  
\newcommand{\gvf}{\varphi}  

\newcommand{\gl}{\lambda}

\newcommand{\gr}{\rho}

\newcommand{\gs}{\sigma}

\newcommand{\cM}{{\mathcal M}}

\newcommand{\cZ}{{\mathcal Z}}

\newcommand{\id}{{\rm id}}

\newcommand{\due}[3]{{}_{{#2 }} {#1}_{{ #3}}\,}    % Zweifachindex

\newcommand{\pl}{\partial}

\newcommand{{\Hl}}{{H^{\ell}}} 
\newcommand{{\mHop}}{{m_{H^{\rm op}}}} 
\newcommand{{\Hop}}{{H^{\rm op}}} 
\newcommand{{\mUop}}{{m_{U^{\rm op}}}} 
\newcommand{{\mUopp}}{{m_{\scriptscriptstyle{U^{\rm op}}}}} 
\newcommand{{\Uop}}{{U^{\rm op}}}
\newcommand{{\mVop}}{{m_{V^{\rm op}}}} 
\newcommand{{\Vop}}{{V^{\rm op}}}  
\newcommand{{\Ae}}{{A^{\rm e}}}
\newcommand{{\Be}}{{B^{\rm e}}}
\newcommand{{\Ree}}{{R^{\rm e}}}
\newcommand{{\He}}{{H^{\rm e}}}
\newcommand{{\Aop}}{{A^{\rm op}}}
\newcommand{{\Aope}}{({A^{\rm op}})^{\rm e}}
\newcommand{{\Aopl}}{{A^{\rm op}_\pl}}

\newcommand{{\Bop}}{{B^{\rm op}}}
\newcommand{{\Bopp}}{{\scriptscriptstyle{{B^{\rm op}}}}}
\newcommand{{\Bope}}{({B^{\rm op}})^{\rm e}}
\newcommand{{\Bpl}}{{B_\pl}}

\newcommand{{\op}}{{{\rm op}}}
\newcommand{{\coop}}{{{\rm coop}}}
\newcommand{{\sop}}{{*^{\rm op}}}
\newcommand{{\co}}{{{\rm co}}}

\newcommand{\umod}{{}_U\cM}                     %  

\newcommand{\comodu}{\cM^U}         
\newcommand{\ucomod}{{}^U\!\cM}

\newcommand{\yd}{{}^U_U\mathbf{YD}}

\newcommand{\ydu}{{}_U \hspace*{-1pt}\mathbf{YD}^{\hspace*{1pt}U}}

\newcommand{\hyd}{{}^H_H\mathbf{YD}}

\newcommand{\ydh}{{}_\mathscr{H} \hspace*{-1pt}\mathbf{YD}^{\hspace*{1pt}\mathscr{H}}}

\newcommand{\tetra}{\mkern 1mu {}^{U}\hskip -7.5pt{}^{\phantom{U}}_U \cM^U_U}

\newcommand{\tetrah}{\mkern 1mu {}^{\mathscr{H}}\hskip -11.5pt{}^{\phantom{\mathscr{H}}}_\mathscr{H} \cM^\mathscr{H}_\mathscr{H}}

\newcommand{\coM}{{}^{\co\,U} \!  M}

\newcommand{\lact}{\smalltriangleright}                  
\newcommand{\ract}{\smalltriangleleft}
\newcommand{\blact}{\blacktriangleright}  
\newcommand{\bract}{\blacktriangleleft}

\newcommand{{\gog}}{{G \rightrightarrows G_0}}

\newcommand{{\rra}}{\rightrightarrows}

\newcommand{{\lra}}{\ \longrightarrow \ }
\newcommand{{\lla}}{\ \longleftarrow \ }
\newcommand{{\lma}}{\ \longmapsto \ }

\def\longleftrightarrows{\buildrel
  {\displaystyle{\relbar\joinrel\longrightarrow}} \over  {\longleftarrow\joinrel\relbar}}

\def\kasten#1{\mathop{\mkern0.7\thinmuskip
\vbox{\hrule
      \hbox{\vrule
            \hskip#1
            \vrule height#1 width 0pt
            \vrule}%
      \hrule}%
\mkern0.7\thinmuskip}}

\newcommand{\bx}{\raisebox{-.5pt}{${\kasten{6pt}}$}}

\newcommand{{\bull}}{{\scriptscriptstyle{\bullet}}}
\newcommand{{\qqquad}}{{\quad\quad\quad}}

\newsavebox{\foobox}

\newcommand{\pmact}{\mbox{ \raisebox{-1pt}{\ding{226}} }}
\newcommand{\mpact}{\mbox{ \raisebox{-1pt}{\ding{227}} }}

\newcommand{\smap}{{\raisebox{0.3pt}{${{\scriptscriptstyle {[+]}}}$}}}
\newcommand{\smam}{{\raisebox{0.3pt}{${{\scriptscriptstyle {[-]}}}$}}}

 \newcommand{\mancino}{{\,\scalebox{0.7}{\rotatebox{90}{\mancone}}\,}}
 \newcommand{\pancino}{\raisebox{4.5pt}{{\,\scalebox{0.7}{\rotatebox{270}{\mancone}}\,}}}



\begin{document}

\title{Hopf bimodules for bialgebroids}

\author{Sophie Chemla}
\author{Niels Kowalzig}

\address{\hskip -7pt S.C.: Sorbonne Universit\'e \& Universit\'e Paris Cit\'e, CNRS, IMJ-PRG, F-75005 Paris, France}
\email{sophie.chemla@imj-prg.fr}

\address{\hskip -7pt N.K.: Dipartimento di Matematica, Universit\`a di Roma Tor Vergata, Via della Ricerca Scientifica~1,
00133 Roma, Italy}
\email{kowalzig@mat.uniroma2.it}

\begin{abstract}
  Generalising a result for Hopf algebras, we not only define the four possible types of Hopf modules in the bialgebroid setting but also
yield the notion of two-sided two-cosided Hopf modules, also known as Hopf bimodules or tetramodules, in this realm.
By explicitly formulating a fundamental theorem for Hopf modules via the concept of Hopf-Galois comodules,  
we prove that the category of Hopf bimodules can be endowed with the structure of a (pre-)braided monoidal category in two different ways, which, in turn, are shown to be both braided monoidally equivalent to the category of Yetter-Drinfel'd modules, that is, to the monoidal centre of the category of left bialgebroid modules or comodules. As an illustration, we discuss relative Hopf bimodules associated to Ehresmann-Schauenburg bialgebroids.
\end{abstract}

\maketitle

\tableofcontents

\section*{Introduction}

Classical Hopf modules, that is, vector spaces that are simultaneously modules and comodules over a Hopf algebra with a (in a certain sense) multiplicative compatibility condition, have been part of Hopf algebra theory right from its beginning. By means of their famous fundamental (or structure) theorem, see \cite[Thm.~4.1.1]{Swe:HA}, they allow to obtain important results such as, for example, the fact that a finite dimensional Hopf algebra (over a field) is always a Frobenius algebra. They also play a r\^ole in the context of homogeneous vector bundles on quantum homogeneous spaces providing insight into representation theory and generalising classical concepts like the Borel-Weil theorem, and are, moreover, of crucial importance in Woronowicz's study of differential calculi over quantum groups \cite{Wor:DCOCMPQG}.

At a next level of complexity, 
one can subsequently consider two-sided and two-cosided Hopf modules over a Hopf algebra (or bialgebra), which are usually termed {\em Hopf bimodules} or sometimes {\em tetramodules}: that is, objects endowed with a left and a right action as well as a left and right coaction, all compatible among each other in a (multiplicative) way.
%The terminology {\em tetramodules} is less common, but in our opinion better describes what it actually is as for the mere definition no Hopf (that is, antipode kind of) structure is needed but only relies on the underlying bialgebra. This, of course, already applies to the term {\em Hopf module}.

Hopf bimodules have been studied under various aspects, and have found successful applications in, just to name a few, homological algebra or in connection to $n$-fold monoidal categories, inasmuch they are the natural choice of coefficients in {\em Gerstenhaber-Schack cohomology}, which is the cohomology theory that governs bialgebra deformations; see, among others, \cite{Bic:HHOHAAFYDROTC, Bic:GSAHHOHA, PanSte:DCFYDMAHB, Sho:TOABFATMC, Tai:CTOHBACP, Tai:IHBCOIDHAAGCOTYP} for interesting results in these directions. Other aspects of Hopf bimodules include coefficients in Hopf-cyclic cohomology and the recent account in \cite{HalKrae:ANSSKLM} on a connection to topological quantum computation by means of the Kitaev model.

\addtocontents{toc}{\SkipTocEntry}

\subsection*{Aims and objectives}

In \cite{Schau:HMAYDM}, Schauenburg discussed categorical aspects of both the category of Hopf modules and that of Hopf bimodules; in particular, their monoidal structure, the existence of (pre-)braidings, as well as the existence of braided monoidal equivalences (in two different ways) of the category of Hopf bimodules to the category of Yetter-Drinfel'd (YD) modules; which, in turn, are related to the monoidal centre construction of the category of modules over the Hopf algebra in question.

The principal aim of this article is to generalise these constructions to {\em (left) bialgebroids}, which can be seen as bialgebras over noncommutative base rings (and that behave to bialgebras as groupoids do to groups), as well as {\em left} and {\em right} Hopf algebroids, by which we mean left bialgebroids that admit certain kinds of invertible Hopf-Galois (or canonical) maps,
see the Appendix for more precise information.

Quite recently in \cite{Han:HBAYDMOHA}, a Hopf bimodule theory was developed in the more restrictive context of {\em full} Hopf algebroids \cite{BoeSzl:HAWBAAIAD}, that is, objects being simultaneously a {\em left} and a {\em right} bialgebroid endowed with an antipode-like morphism  intertwining these two structures. However, our Definition \ref{tetra} for Hopf bimodules, see below, is based on the more general (and weaker) notion of  bialgebroid only, and hence not equivalent to that appearing in {\em op.~cit.} (as, in general, there might be no underlying right bialgebroid structure or antipode needed there to make the definition work).

\addtocontents{toc}{\SkipTocEntry}

\subsection*{New definitions and main results}

The definition of the four possible variants of {\em Hopf modules} over a (left) bialgebroid $(U,A)$, that is, the definition of an object $M$ equipped with simultaneously a (left or right) $U$-action (denoted by juxtaposition) and a (left or right) {$U$-coaction} (denoted by $\gl_M$ or $\rho_M$) in a compatible way, requires (in particular in the case of right $U$-actions) more technical attention than in the Hopf algebra case, in order to give a well-defined sense to expressions like $\rho_M(mu) = \rho_M(m)\gD(u)$ or $\gl_M(mu) = \gl_M(m)\gD(u)$, where $\gD$ denotes the coproduct in $U$, since bialgebroids are monoids and comonoids in different kinds of (even nonsymmetric) categories of (bi)modules over the underlying base algebra $A$ resp.\ over $\Ae := A \otimes \Aop$.
As a consequence, a Hopf module a priori comes with four different $A$-actions of which two are asked to coincide, and a third one is {\em required} to commute with the (right) $U$-action. Full details can be found in the new (and symmetrically formulated) Definitions \ref{genzano} \& \ref{frascati}.

A {\em Hopf bimodule} or {\em tetramodule} is then an $\Ae$-bimodule that is, similar to the Hopf algebra case, simultaneously a Hopf module in all possible four different senses, see the equally new Definition \ref{tetra}.
For this, no Hopf structure or antipode of any kind is needed.

Next, we recapitulate the definition of what we call a (left or right) Hopf-Galois comodule of which the definition appeared in \cite{Che:ITFLHLB} before. These are (right or left) comodules endowed with a certain invertible Hopf-Galois, or canonical map, that is, comodules which bear a sort of Hopf structure themselves, see Definition \ref{minestra} for more precise information. This notion allows us to compactly formulate a {\em Fundamental} (or {\em Structure}) {\em Theorem} for Hopf modules in the bialgebroid setting, that is, any Hopf module can be constructed (in a ``free'' way) from its coinvariants. In analogy to the Hopf algebra case in \cite{Schau:HMAYDM}, this is the base of the stronger statement that the category of Hopf bimodules is equivalent  to that of Yetter-Drinfel'd modules, which, in turn, is equivalent (see \cite[Prop.~4.4]{Schau:DADOQGHA}) to the category given by the (right weak) monoidal centre in the category ${}_U \cM$ of left $U$-modules, or, alternatively, to the category given by the (left weak) monoidal centre in the category $\cM^U$ of right $U$-comodules, as we prove in Proposition \ref{endlichwaermerjetzt?}. In Theorems \ref{spinoza} \& \ref{rhodia}, we prove (see the main text for notation and details):

\begin{theorem*}
  Let $(U,A)$ be
a right Hopf algebroid, that is, a left bialgebroid endowed with a (certain) invertible Hopf-Galois map,
and let the right $A$-module $U_\ract$ be flat.
\begin{enumerate}
 \compactlist{50}
\item Then the functors
  between the category of right-left Hopf modules and left $A$-modules
  $$
  {}^{\co\,U} \! (-) \colon {}^U\mkern -5mu\cM_U
\quad \raisebox{-3pt}{$\longleftrightarrows$} \quad {}_A \cM \ \colon \!  U \otimes_A - 
  $$
are mutually adjoint and  establish an equivalence of categories.
\item
  The same functors can also be defined between the category of Hopf bimodules and that of (left-right) Yetter-Drinfel'd modules such that
$$
  {}^{\co\,U } \! (-) \colon {}_U^{\mkern 2mu U}\mkern -1mu\cM^U_U
\quad \raisebox{-3pt}{$\longleftrightarrows$} \quad \ydu \ \colon \!  U_\ract \otimes_A - 
$$
are again mutually adjoint functors and  again establish an equivalence of categories.
\end{enumerate}
\end{theorem*}

\noindent Similar statements can be made with respect to the other three possible kinds of Hopf modules in its first and second part, and also with respect to the category $\yd$ of left-left YD modules in its second.
As a next goal, we want to enhance the above categorical equivalence to an equivalence of monoidal categories, and, in particular, of {\em braided} monoidal categories. Combining and concentrating the statements of Lemma \ref{exposition}, Proposition \ref{resistente}, Theorem \ref{spielzeit}, and of Theorem \ref{buddha}, we obtain (see again the main text for notation and details):

\begin{theorem*}
  Let $(U,A)$ be a left bialgebroid.
  \begin{enumerate}
    \compactlist{50}
    \item
The category
$\tetra$ of Hopf bimodules forms a monoidal category with respect to the tensor product over $U$ (with monoidal unit $U$).
\item
  If both $U_\ract$ and $\due U \lact {}$ are $A$-flat, then 
  the category $\tetra$ of Hopf bimodules 
  with respect to the cotensor product over $U$ (with monoidal unit $U$) forms a monoidal category as well. 
%% \item
%%   If both $U_\ract$ and $\due U \lact {}$ are projective over $A$ and if the left bialgebroid $(U,A)$ is a left Hopf algebroid, that is, equippped with a (certain) invertible Hopf-Galois map, then
%%   these are monoidally equivalent, and, in particular, equivalent to the category of left-right Yetter-Drinfel'd modules, {\em i.e.}, one has a chain 
%%   $$
%% \big(\tetra, \otimes_{\mkern 1mu U} , U\big) \ \simeq \ \big(\tetra, \bx_U, U\big) \ \simeq \ \big(\ydu, \otimes_A, A \big)
%%   $$
%%   of monoidal equivalences of categories.
\item
  If both $U_\ract$ and $\due U \lact {}$ are $A$-flat and if $(U,A)$ is both a left and right Hopf algebroid, that is, a left bialgebroid equipped with two different kinds of invertible Hopf-Galois maps, then the two categories above are monoidally equivalent and, in particular, admit braidings that fit into
a chain of equivalences
  $$
\qquad\ \big(\tetra, \otimes_{\mkern 1mu U} , U, \tau\big) \ \simeq \ \big(\tetra, \bx_U, U, \tau'\big) \ \simeq \ \big(\ydu, \otimes_A, A, \gs \big) \ \simeq \ \big(\yd, \otimes_A, A, \gs' \big)
  $$
of braided monoidal categories.
  \end{enumerate}
\end{theorem*}

In Section \ref{exa}, we discuss a class of examples given by
so-called {\em Ehresmann-Schauenburg bialgebroids} as introduced in \cite{Schau:BONRAASTFHB}, or rather a left coinvariant version of it: for a Hopf algebra $H$ over a field $k$ with invertible antipode and a faithfully flat Hopf-Galois extension $B \subseteq A$, the space ${}^{\co H}\! (A \otimes A) =: \mathscr{H}$ can be given the structure of a left bialgebroid over $B$, and also of a right Hopf algebroid; in Proposition \ref{hrensko}, we will give a sufficient criterion for this right Hopf algebroid to become a full Hopf algebroid, admitting an antipode. Moreover, as in {\em op.~cit.}, we introduce the category of {\em relative Hopf bimodules} ${}^{\,H}_A \! \cM^\mathscr{H}_A$, which in Theorem \ref{weich} and Corollaries \ref{weich2} \& \ref{weich3} allows us to prove:

\begin{theorem*}
  Let $H$ be a Hopf algebra over a field $k$ with invertible antipode, and let $B \subseteq A$ be a faithfully flat Hopf-Galois extension. Then one has equivalences
  $$
\tetrah \simeq  \ydh \simeq {}^{\,H}_A \! \cM^\mathscr{H}_A \simeq \hyd
  $$
  of monoidal categories.
\end{theorem*}

To conclude, the appendix contains a summary on basic results of left bialgebroids, their possible left or right Hopf structure and appurtenant terminology, plus modules, comodules, and YD modules over (left) bialgebroids, as well as all notation used throughout the text.

\addtocontents{toc}{\SkipTocEntry}

\subsection*{Notation}
Let $k$ be a commutative ring usually referred to as ground ring, which might be a field and of characteristic zero if need be.
%As customary, unadorned tensor products or $\Hom$s refer to those over $k$.
We also use multiple kinds of Sweedler notation, with left comodules being indicated by round brackets, right comodules by square brackets. Have a look as well at the appendix for more bialgebroid related notation.

\addtocontents{toc}{\SkipTocEntry}

\subsection*{Acknowledgements}
Partially supported by the MIUR Excellence Department Project MatMod@TOV
(CUP:E83C23000330006).
N.\ K.\ is a  member of the {\em Grup\-po} {\em Na\-zio\-na\-le} {\em per le Strutture Algebriche, Geometriche e le loro
   Applicazioni} (GNSAGA-INdAM).
He wants to thank in particular the {\em Institut de Math\'ematiques de Jussieu -- Paris Rive Gauche}, where part of this work has been achieved, for hospitality and support through the {\em CNRS}.
Both authors would like to thank Peter Schauenburg for helpful discussions.

 \section{Hopf modules, Hopf bimodules, and Hopf-Galois comodules}
 \label{one}

 For a comodule over a bialgebroid $(U,A)$, be it left or right (see \S\ref{schleifmaschine} for a summary), it is natural to study those which, in addition, admit a compatible (left or right) $U$-action. There are various types of compatibility conditions between action and coaction, the possibly simplest being a kind of multiplicativity that defines {\em Hopf modules} and related, but with a richer structure, {\em Hopf bimodules} (also known as {\em tetramodules}), while {\em Yetter-Drinfel'd modules} come with a very different kind of compatibility.

 \subsection{Hopf modules over bialgebroids}
The four kinds of Hopf modules for a bialgebroid formally look the same as they do for Hopf algebras \cite[\S4.1]{Swe:HA}, but depending on the specific kind and due to the noncommutativity and noncentrality of the base rings, some extra conditions are needed to make the appearing expressions well-defined. More precisely:

\begin{definition}
  \label{genzano}
  Let $(U, A)$ be a left bialgebroid and let $M \in \comodu$ be a right $U$-comodule, with its {\em natural} right resp.\ {\em induced} left $A$-action 
  $
  a \otimes m \otimes  a'  \mapsto ama'
  $
denoted by juxtaposition. %,  for $a, a'  \in A$ and $m \in M$.
  \begin{enumerate}
    \compactlist{99}
    \item
  \label{genzano2}
If $M$ is, moreover, a left $U$-module such that
  \begin{enumerate}
\compactlist{99}
  \item
    the natural right $A$-action on $M$ as a right $U$-comodule {\em coincides} with the right $A$-action induced by the left $U$-action, that is, $ma = m \ract a$;
\item
  the induced left $A$-action on $M$ as a right $U$-comodule {\em commutes} with the left $U$-action on $M$ (and hence with the left $A$-action induced by the left $U$-action);
\item
  the equation
  \begin{equation}
    \label{accedi1}
    \rho_M(um)=\gD(u) \rho_M(m)
    \end{equation}
  holds as a product in $M \times_A \due U \lact {}$,
  where $\gD$ denotes the coproduct in $U$, 
\end{enumerate}
  then $M$ is called a {\em left-right Hopf module} over $U$.

  \noindent The corresponding category will be denoted by ${}_U\cM^U$.
\item
  \label{genzano1}
  If $M$ is, moreover, a right $U$-module such that
  \begin{enumerate}
    \compactlist{99}
\item
  the induced left $A$-action on $M$ as a right $U$-comodule {\em coincides} with the left $A$-action induced by the right $U$-action, that is,
  $am = a \blact m$;
\item
the natural right $A$-action on $M$ as a right $U$-comodule {\em commutes} with the right $U$-action on $M$ (and hence with the right $A$-action induced by the right $U$-action);
\item
  the equation
  \begin{equation}
    \label{accedi2}
    \rho_M(m u)=\rho_M(m)\Delta(u)
    \end{equation}
holds as a product in $M \times_A \due U \lact {}$, where $\gD$ denotes the coproduct in $U$, 
  \end{enumerate}
  then $M$ is called a {\em right-right Hopf module} over $U$.

  \noindent The corresponding category will be denoted by $\cM^U_U$.
\end{enumerate}
  \end{definition}

\begin{rem}
  Observe that (b) in part (i) is redundant as this directly follows from Eq.~\eqref{accedi1}, while this is not so for the respective statements in part (ii), where (b) is needed to give a sense to Eq.~\eqref{accedi2}.
  For the sake of a better symmetry we formulated it this way, but it shows a certain structural difference between left-right and right-right Hopf modules, the problem arising from the non-monoidality of the category of right $U$-modules.
    \end{rem}

\begin{definition}
  \label{frascati}
  Let $(U, A)$ be a left bialgebroid and let $M \in \ucomod$ be a left $U$-comodule, with its {\em natural} left resp.\ {\em induced} right $A$-action 
  $
  a \otimes m \otimes  a' \mapsto a \cdot m \cdot a'
  $
denoted by a central dot.
%  for $a,a' \in A$ and $m \in M$.
  \begin{enumerate}
\compactlist{99}
\item
  \label{frascati1}
  If $M$ is, moreover, a right $U$-module such that
  \begin{enumerate}
    \compactlist{99}
       \item
the induced right $A$-action on $M$ as a left $U$-comodule {\em coincides} with the right $A$-action induced by the right $U$-action, that is, $m \cdot a= m\bract a$;
\item
  the natural left $A$-action on $M$ as a left $U$-comodule {\em commutes} with the right $U$-action on $M$ (and hence with the left $A$-action induced by the right $U$-action);
\item
  the equation
    \begin{equation}
    \label{accedi3}
    \lambda_M(m u)=\lambda_M(m)\Delta(u)
\end{equation}
    holds as a product in $U_\ract \times_A M$, where $\gD$ denotes the coproduct in $U$, 
  \end{enumerate}
then
we call $M$ a \textit{right-left Hopf module} over $U$.

\noindent The corresponding category will be denoted by ${}^U\mkern -5mu\cM_U$.
\item
  \label{frascati2}
If $M$ is, moreover, a left $U$-module such that
\begin{enumerate}
  \compactlist{99}
\item
  the natural left $A$-action on $M$ as a left $U$-comodule {\em coincides} with the left $A$-action induced by the left $U$-action, that is, 
  $a \cdot m = a \lact m$;
\item
  the induced right $A$-action on $M$ as a left $U$-comodule {\em commutes} with the left $U$-action on $M$ (and hence with the right $A$-action induced by the left $U$-action);
\item
  the equation
    \begin{equation}
    \label{accedi4}
    \lambda_M(um)=\gD(u) \lambda_M(m)
\end{equation}
    holds as a product in $U_\ract \times_A M$, where $\gD$ denotes the coproduct in $U$, 
\end{enumerate}
then
we call $M$ a \textit{left-left Hopf module} over $U$.

\noindent The corresponding category will be denoted by
${}_U^{\mkern 2mu U}\mkern -1mu\cM$.
\end{enumerate}
  \end{definition}

\begin{rem}
  \label{roslagen}
Again, observe that this time (b) in part (ii) is redundant since this directly follows from Eq.~\eqref{accedi4}, while this is not so for the respective statements in part (i), where (b) is needed to give an actual sense to Eq.~\eqref{accedi3}.
  For the sake of a better symmetry we formulated it this way, but it again shows a certain structural difference between right-left and left-left Hopf modules, the problem again arising from the non-monoidality of the category of right $U$-modules, as above.
\end{rem}

\begin{example}
The total space $U$ of a left bialgebroid is a Hopf module over itself in all four possible senses, by using its multiplication and comultiplication, while the base algebra $A$, despite being endowed at least with a left $U$-action via the counit as well as a left and right $U$-coaction via the source resp.\ the target map, is not.
  \end{example}

\subsection{Hopf bimodules over bialgebroids}
Intuitively quite clear, the notion of Hopf bimodule or tetramodule over a left bialgebroid consists in putting all four parts in Definitions \ref{genzano} \& \ref{frascati} together:

\begin{definition}
  \label{tetra}
A {\em Hopf bimodule} or {\em tetramodule} over a left bialgebroid $(U,A)$ is a $U$-bimodule and a $U$-bicomodule that is simultaneously an object in $ {}_U\cM^U$, $ \cM^U_U$, ${}_U^{\mkern 2mu U}\mkern -1mu\cM$, as well as ${}^U\mkern -5mu\cM_U$. The corresponding category will be denoted by ${}_U^{\mkern 2mu U}\mkern -1mu\cM^U_U$.
  \end{definition}

\begin{rem}
  Hence, by Definitions \ref{genzano} \& \ref{frascati}, a Hopf bimodule $M$ has four $A$-module structures, or rather eight, of which four pairwise coincide: with the notation introduced above that $A$-actions denoted by juxtaposition originate from the right $U$-coaction and those denoted by central dots $\cdot$ from the left one, one therefore obtains
  \begin{equation}
    \label{rain}
    \begin{array}{lr}
      ma = m \ract a = t(a)m,
      &
      am = a \blact m = mt(a),
\\
      m\cdot a = m \bract a = m s(a),
      &
      a \cdot m = a \lact m = s(a)m,
  \end{array}
  \end{equation}
  for $a \in A$ and $m \in M$, where $M$ is a Hopf bimodule over $U$.
\end{rem}

In general, let $X$ and $Y$ be two $U$-bimodules. Then their Takeuchi-Sweedler product $X \times_A Y$ is a $U$-bimodule by means of the usual diagonal action from left and right:
\begin{equation}
  \label{fog}
u \mancino (x \times_A y) := u_{(1)} x \times_A u_{(2)} y, \qquad
(x \times_A y) \pancino u := xu_{(1)}  \times_A yu_{(2)}.
\end{equation}
for $x \in X, y \in Y$, $u \in U$.
Then observe that a Hopf bimodule $M$ is, in particular, a bicomodule~with right and left coactions $\rho_M \colon M \to M \times_A U$ and $\gl_M \colon M \to U \times_A M$ as in Definitions \ref{genzano} \& \ref{frascati}. Hence, the four conditions \eqref{accedi1}--\eqref{accedi4} for a Hopf bimodule over $U$ can be rewritten the following way:
\begin{eqnarray}
  \label{sunshine1}
  \rho_M(um)= u \mancino \rho_M(m),
&&
  \rho_M(m u)=\rho_M(m)\pancino u,
  \\
    \label{sunshine2}
  \gl_M(m u)=\gl_M(m)\pancino u,
  &&
  \gl_M(um)=u \mancino \gl_M(m).
\end{eqnarray}

\noindent This proves (or rather tautologically states) the following:

\begin{prop}
\label{snow}
In a Hopf bimodule, the left and right $U$-actions are left and right $U$-colinear. Said fancier, a Hopf bimodule is a $U$-bimodule in the monoidal category
$\big({}^U \hskip -3pt\cM^U, \otimes_{\Ae}, \Ae\big)$ of $U$-bicomodules.
\end{prop}

\begin{rem}
In striking contrast to the bialgebra case, although Eqs.~\eqref{sunshine1} \& \eqref{sunshine2} equally express left and right $U$-linearity of the left and right $U$-coactions, this cannot be reformulated by stating that a Hopf bimodule is a $U$-bicomodule in the category ${}_U\hskip -1pt\cM_U$ of $U$-bimodules, the problem being that the latter is, in general, not monoidal.
\end{rem}

With Definition \ref{tetra} above, it is quite straightforward to prove that the category ${}_U^{\mkern 2mu U}\mkern -1mu\cM^U_U$ of tet\-ra\-modules is equivalent to the category $\ydu$ of right-left YD modules (or the category $\yd$ of left-left YD modules) as in the Hopf algebra case \cite{Schau:HMAYDM}. This will be the content~of~\S\ref{tetrayd}.

\subsection{Hopf-Galois comodules}
\label{lostbuddha}

For the various categorical equivalences involving Hopf modules, Hopf bimodules, and Yetter-Drinfel'd modules that will be dealt with in the subsequent Section \ref{tetrayd},
we need yet another notion, which can be understood as transferring the Hopf structure to the comodule in question itself: a left $U$-comodule $(M, \gl_M)$ with left coaction denoted
$\gl_M \colon M \to U_\ract \otimes_A M, \ m \mapsto m_{(-1)} \otimes_A m_{(0)}$, permits to define a
{\em Galois map} as
\begin{equation}
	\label{duracell}
	\alpha_M \colon M\otimes_AU\to U\otimes_AM, \quad m\otimes_A u
        \mapsto m_{(-1)}u \otimes_A m_{(0)},
\end{equation}
of which the possible invertibility yields the following notion:

\begin{definition}
\label{minestra}
If, for a left $U$-comodule $M$ over a left bialgebroid $(U, A)$,
the map $\ga_M$ from \eqref{duracell} is invertible, then we
call
$M$ a {\em (right) Hopf-Galois comodule} over $U$ and
write
$$
	m_{\smap }\otimes_A m_{\smam }:= \alpha_M^{-1}(1_U\otimes_A m)
$$
for the {\em translation map} $M \to M \otimes_A U$ on $M$.
\end{definition}

The terminology {\em right} in the above definition will be explained in a moment.
In case of such a Hopf-Galois comodule, {\em i.e.}, if the Galois map of a left $U$-comodule $M$ is invertible, the subsequent equalities that appeared in \cite[Prop.~3.0.6]{Che:ITFLHLB}
\begin{align}
	m_{\smap }\otimes_A m_{\smam }&\in M \times_AU,\label{Mch1}
	\\
        \label{Mch2}
	m_{\smap (-1)} m_{\smam }\otimes_Am_{\smap (0)}
	&=1\otimes_Am,
	\\
        \label{Mch3}
	m_{(0)\smap }\otimes_A m_{(0)\smam } m_{(-1)}
	&= m \otimes_A 1,
	\\
\label{Mch4}
        m_{\smap \smap }\otimes_A m_{\smap \smam }\otimes_A m_{\smam }
	&=m_{\smap }\otimes_Am_{\smam (1)}\otimes_Am_{\smam (2)},
        	\\
\label{Mch5}
        (a \cdot m \cdot a')_{\smap }\otimes_{A}(a \cdot m \cdot a')_{\smam }
	&=m_{\smap }\otimes_{A} a \blact m_{\smam } \ract a',
		\\
\label{Mch6}
        m_{\smap } \cdot \varepsilon(m_{\smam })
	&=m,
        \end{align}
hold true for $a,a' \in A$ and $m \in M$,
where in \eqref{Mch1} the Takeuchi-Sweedler product
$$
M \times_A U = \big\{ \textstyle\sum_i m_i \otimes u_i \in M \otimes_A \due U \lact {} \mid  \sum_i a \cdot m_i \otimes u_i = \sum_i m_i \otimes u_i \bract a, \ \forall a \in A   \big\}
$$
is meant.
From \eqref{Mch4} and \eqref{Mch6} one deduces that $\rho_M := \ga^{-1}_M (- \otimes_A 1_U)$,
that is, the map
\begin{equation*}
	%\label{fabercastell}
	\rho_M \colon M \to M \otimes_A U, \quad m \mapsto m_{\smap } \otimes_A m_{\smam }
\end{equation*}
turns the left $U$-comodule $M$ into a right one, which illuminates why we baptised it a {\em right} Hopf-Galois comodule.
Invertibility of the Galois map of a left $U$-comodule is granted in case $(U,A)$ is a right Hopf algebroid (over a left bialgebroid). More precisely, in \cite[Lem.~2.9]{KowWeb:TSEACFHSOHA} the following statement was proven:

\begin{lemma}
	A left bialgebroid $(U, A)$ is a right Hopf algebroid if and only if any left $U$-comodule $M$ is a right Hopf-Galois comodule.
\end{lemma}

\noindent In this case, the inverse of the Galois map is given by
$$
U \otimes_A M \to M \otimes_A U,
\quad
u \otimes_A m
\mapsto
\varepsilon(m_{(-1)\smap })\cdot  m_{(0)} \otimes_A m_{(-1)\smam } u,
$$
that is, the induced right $U$-coaction on $M$ becomes 
\begin{equation}
	\label{fabercastell0}
	m_{\smap } \otimes_A m_{\smam } :=  \varepsilon(m_{(-1)\smap }) \cdot m_{(0)} \otimes_A m_{(-1)\smam },
\end{equation}
see {\em loc.~cit.} for details on how to obtain this coaction.

\begin{lemma}
\label{minestrone}
Let $(U, A)$ be a right Hopf algebroid (over a left bialgebroid) and $M\in {}^U\hskip -2.5pt\cM_U$ be a left-right Hopf module. Then the identity
\begin{equation*}
\label{scalea}
  \rho_M(mu)
= (mu)_{\smap } \otimes_A (mu)_{\smam }
  =  m_{\smap } u_{\smap } \otimes_A u_{\smam } m_{\smam }  
  \end{equation*}
holds for all $m \in M$ and $u \in U$ with respect to the right $U$-coaction
\eqref{fabercastell0}.
\end{lemma}

\begin{proof}
  By means of Definition \ref{frascati}\,(i) along with Eq.~\eqref{fabercastell0}, one computes:
  \begin{equation*}
    \begin{split}
      (mu)_{\smap } \otimes_A (mu)_{\smam }
      &=
      \varepsilon\big((mu)_{(-1)\smap }\big) \cdot (mu)_{(0)} \otimes_A (mu)_{(-1)\smam }
      \\
      &=
      \varepsilon(m_{(-1)\smap }u_{(1)\smap }) \cdot (m_{(0)}u_{(2)}) \otimes_A u_{(1)\smam } m_{(-1)\smam }
      \\
      &=
      \big(\varepsilon(m_{(-1)\smap }u_{(1)\smap }) \cdot m_{(0)}\big) u_{(2)} \otimes_A u_{(1)\smam } m_{(-1)\smam }
 \\
      &=
      \big(\varepsilon(m_{(-1)\smap }u_{\smap(1)}) \cdot m_{(0)}\big) u_{\smap(2)} \otimes_A u_{\smam } m_{(-1)\smam }
\\
      &=
      \pig(\varepsilon\big(\gve(u_{\smap(1)}) \blact m_{(-1)\smap }\big) \cdot m_{(0)}\pig) u_{\smap(2)} \otimes_A u_{\smam } m_{(-1)\smam }
      \\
      &=
    \big(  \varepsilon(m_{(-1)\smap }) \cdot m_{(0)} \cdot \gve(u_{\smap(1)}) \big) u_{\smap(2)} \otimes_A u_{\smam } m_{(-1)\smam }
      \\
      &=
    \big(  \varepsilon(m_{(-1)\smap }) \cdot m_{(0)}\big)\big(\gve(u_{\smap(1)}) \lact u_{\smap(2)}\big) \otimes_A u_{\smam } m_{(-1)\smam }
    \\
      &=
     m_{\smap }  u_{\smap } \otimes_A u_{\smam } m_{\smam },
          \end{split}
  \end{equation*}
  where in the second step we used
both  Definition \ref{frascati}, part\,(i)\,(c), and \eqref{Tch6}, in step three Definition \ref{frascati}, part\,(i)\,(b), 
then Eq.~\eqref{Tch4} in the fourth step,
the  generalised-character-like properties
  of a bialgebroid counit in the fifth,
the 
Takeuchi property \eqref{ha3} for a left $U$-comodule
along with \eqref{Tch9} in the sixth,
Definition \ref{frascati}, part\,(i)\,(c),
and again 
Definition \ref{frascati}, part\,(i)\,(b), in the seventh, and finally
 counitality of $U$ along with \eqref{fabercastell0} in the eighth.
 \end{proof}
  
Of course, the same game could be played if one starts from a right $U$-comodule $(N, \rho_N)$ with right coaction denoted $\rho_N \colon N \to N \otimes_A \due U \lact {}, \ n \mapsto n_{[0]} \otimes_A n_{[1]}$ to define a Galois map by
\begin{equation}
  \label{duracell2}
  \alpha_N \colon U_\ract \otimes_A N \to N \otimes_A \due U \lact {}, \quad u \otimes_A n
   \mapsto n_{[0]}  \otimes_A n_{[1]}u,
  \end{equation}
of which the possible invertibility yields the notion of {\em left} Hopf-Galois comodule over $U$ by asking for the invertibility of \eqref{duracell2} and therefore writing
$$
n_- \otimes_A n_+ := \alpha_N^{-1}(n \otimes_A 1)
$$
for the translation map $N \to U_\ract \otimes_A N$ on $N$. One analogously obtains the following identities for $a ,a'  \in A$ and $n \in N$:
\reversemarginpar
  \begin{align}
\label{Nch1}
    n_- \otimes_A n_+ &\in U \times_A N,
    \\
    \label{Nch2}
n_{+ [0]} \otimes_A n_{+ [1]} n_-
&= n \otimes_A 1,
\\
\label{Nch3}
 n_{[0]- } n_{[1]} \otimes_A n_{[0] + }
&= 1 \otimes_A n,
\\
\label{Nch4}
n_{-}\otimes_A n_{+- }\otimes_A n_{++}
&=n_{-(1)}\otimes_A n_{- (2)}\otimes_A n_{+},
\\
\label{Nch5}
(ana' )_- \otimes_A (ana' )_+
&= a \lact n_- \bract a' \otimes_A n_+,
\\
\label{Nch6}
\varepsilon(n_-) n_+
&=m,
  \end{align}
  where in \eqref{Nch1}  the {\em Takeuchi-Sweedler} subspace (or product) $U \times_A N$ is defined as
%\begin{small}
  $$
  U \times_A N := \big \{\Sum_i u_i \otimes_A n_i \in U \otimes_A N \mid
\Sum_i  u_i \otimes_A n_ia = \Sum_i a \blact u_i \otimes_A n_i, \forall  a \in A
  \big\}.
$$
%\end{small}
Similar to the considerations above, it follows from \eqref{Nch4} and \eqref{Nch6} that the translation map
\begin{equation}
	\label{fabercastell2}
\gl_N := \ga^{-1}_N (- \otimes_A 1) \colon N \to U_\ract \otimes_A N, \quad n \mapsto n_- \otimes_A n_+
\end{equation}
turns the right $U$-comodule $(N, \rho_N)$ we started with into a left one, which justifies the terminology {\em left} Hopf-Galois comodule.
Again, invertibility of the Galois map of a left $U$-comodule is automatic if $U$ is a {\em left} Hopf algebroid, or explicitly 
$$
N \otimes_A \due U \lact {} \to U_\ract \otimes_A N,
\quad
n \otimes_A u
\mapsto
n_{[1]- } u \otimes_A n_{[0]} \varepsilon(n_{[1]+})
$$
such that
\begin{equation}
  \label{fabercastell3}
\gl_N(n) := n_- \otimes_A n_+ := n_{[1]- } \otimes_A n_{[0]} \varepsilon(n_{[1]+})
\end{equation}
is a left $U$-coaction on $N$. A result analogous to Lemma \ref{minestrone} then states that for a left Hopf algebroid $(U,A)$ and a {\em right-right} Hopf module $N$, one obtains the identity
\begin{equation}
  \label{scalea2}
  \gl_N(nu)
= (nu)_- \otimes_A (nu)_+
=
u_- n_-  \otimes_A
n_+ u_+ 
  \end{equation}
with respect to the left $U$-coaction
\eqref{fabercastell2},
for all $n \in N$ and $u \in U$.

Next, for Hopf modules (of any kind) with invertible Hopf-Galois map let us formulate and prove a {\em Fundamental Theorem of Hopf modules}, which reduces to the already mentioned fundamental theorem of Hopf modules \cite[Thm.~4.1.1]{Swe:HA} in the case of Hopf algebras. In order to discuss this,
let us introduce first for a left $U$-comodule $(M,\gl_M)$ by
 \begin{equation}
   \label{waterman}
   {}^{\mathrm{co} \, U}\! M:= \{m\in M \mid \gl_M(m) = 1 \otimes_A m \}
 \end{equation}
 the $k$-module of {\em (left) $U$-coinvariant} elements.
 {\em Right} coinvariants $N^{\mkern 1mu \mathrm{co} \, U}$ for a right $U$-co\-mod\-ule $N$ are defined analogously. In the bialgebroid setting, we can then state:   

\begin{theorem}
  \label{spinoza}
Let $(U,A)$ be a left bialgebroid such that $U_\ract$ is flat as a right module over $A$, and
  let $M \in {}^U\mkern -5mu\cM_U$ be a right-left Hopf module that is also a right Hopf-Galois comodule in the sense of Definition \ref{minestra}.
  Then
\begin{enumerate}
 \compactlist{50}
\item
  $m_{\smap }m_{\smam }\in \coM$ for all $m\in M$;
\vskip 1.2pt
\item
 $\coM$ becomes a left $A$-module as well via $a \cdot m = a\blact {m}= {m}t(a)$ for $m \in \coM$ and $a \in A$;
  \vskip 1.2pt
\item if $(U,A)$ is a right Hopf algebroid (over a left bialgebroid), then
the functors
  $$
  {}^{\co\,U} \! (-) \colon {}^U\mkern -5mu\cM_U
\quad \raisebox{-3pt}{$\longleftrightarrows$} \quad {}_A \cM \ \colon \!  U_\ract \otimes_A - 
  $$
are mutually adjoint and induce an equivalence of categories, with counit given by the isomorphism 
  $$
\xi \colon  U_\ract  \otimes_A \coM \to  M, \quad u\otimes_A {m}\mapsto {m}u, 
$$
in the category ${}^U\mkern -5mu\cM_U$ of right-left Hopf modules of which the inverse reads
$$
\xi^{-1} \colon
M \to U_\ract  \otimes_A \coM, \quad
m\mapsto m_{(-1)}\otimes_A m_{(0)\smap } m_{(0)\smam }.
$$
The
unit of the adjunction is given by the isomorphism
$
\eta \colon X \to {}^{\co\,U} \! (U_\ract \otimes_A X),
\
x \mapsto 1_U \otimes_A x,
$
of which the inverse reads
$
%\eta^{-1} \colon  {}^{\co\,U} \! (U\otimes_A V)
%\to {}_A\cM,
%\quad
u \otimes_A x \mapsto \varepsilon(u) \cdot x.
$
\end{enumerate}
\end{theorem}

\begin{proof}%[Proof of Theorem \ref{spinoza}]
    That $a \cdot m = a \blact m  = m t(a) \in \coM$ if $m \in \coM$ is obvious from the defining condition $\gl_M(mu) = \gl_M(m)\gD(u)$ in \eqref{accedi3} for a right-left Hopf module.
  Equally direct from this condition is the statement that $m_\smap m_\smam \in \coM$ for all $ m \in M$: one has
  \begin{equation*}
    \begin{split}
      \gl_M(m_\smap m_\smam) = \gl_M(m_\smap)\gD(m_\smam)
      &= m_{\smap (-1)} m_{\smam(1)} \otimes_A m_{\smap(0)} m_{\smam(2)}
      \\
      &= m_{\smap\smap (-1)} m_{\smap\smam} \otimes_A m_{\smap\smap(0)} m_{\smam}
      \\
      &= 1 \otimes_A m_\smap m_\smam,
          \end{split}
    \end{equation*}
  where we used Eqs.~\eqref{Mch4} resp.\ \eqref{Mch2} in the third resp.\ fourth step. We are left with showing that $\xi$ is invertible with inverse given by the map we already called $\xi^{-1}$: one computes
  $$
  (\xi \circ \xi^{-1})(m) = \xi(m_{(-1)} \otimes_A m_{(0)\smap} m_{(0)\smam} )
  =
  m_{(0)\smap} m_{(0)\smam}  m_{(-1)} = m,
  $$
 where we used \eqref{Mch3}.
  On the other hand, using that $M \in {}^U\mkern -5mu\cM_U$ is a right-left Hopf module in the second step, and  
Lemma \ref{minestrone}
  in the third, while  Eqs.~\eqref{Tch7} \& \eqref{Mch5} along with Definition \ref{frascati}, part (ii)\,(a), are used in step four, we compute:  
  \begin{equation*}
    \begin{split}
      (\xi^{-1} \circ \xi)(m \otimes_A u) = \xi^{-1}(mu)
      &=
      (mu)_{(-1)} \otimes_A (mu)_{(0)\smap} (mu)_{(0)\smam} 
      \\
      &
      =
      m_{(-1)}u_{(1)} \otimes_A (m_{(0)}u_{(2)})_\smap (m_{(0)}u_{(2)})_\smam
      \\
       &=
       m_{(-1)}u_{(1)} \otimes_A m_{(0)\smap}u_{(2)\smap} u_{(2)\smam} m_{(0)\smam} 
 \\
 &=
       m_{(-1)}u_{(1)} \otimes_A \big(m_{(0)}\cdot \gve(u_{(2)})\big)_\smap \big(m_{(0)}\cdot \gve(u_{(2)})\big)_\smam
 \\
 &=
       m_{(-1)}u \otimes_A m_{(0)\smap} m_{(0)\smam}
 \\
&=     u \otimes_A m_{\smap} m_{\smam}
 \\
&=       u \otimes_A m,
    \end{split}
  \end{equation*}
  where moreover in step five the Takeuchi property of the left coaction $\gl_M$ along with counitality of the coproduct was used, while in step six the fact that $m \in  \coM$ was needed, while in step seven we used, in addition, that
  \begin{equation}
    \label{schnief}
    m_\smap \otimes_A m_\smam = m \otimes_A 1_U, \quad \forall\, m \in  \coM
\end{equation}
holds true.
  Hence, $\xi$ and $\xi^{-1}$ are mutually inverse. The last statement follows from the preceding ones once established that the respective morphisms are transformed into another, which is left to the reader.
     \end{proof}

\begin{rem}
  \label{nunzia}
For a proof of the above theorem in analogous situations such as, for example, the category ${}_U^{\mkern 2mu U}\mkern -1mu\cM$ of left-left Hopf modules, one rather uses homological arguments instead of explicit maps, see \cite[Thm.~5.0.2]{Che:ITFLHLB}. For full Hopf algebroids, equipped with a sort of antipode that mediates between a left and a right bialgebroid structure, a fundamental theorem of similar kind has been proven in \cite[Thm.~4.2]{Boe:ITFHA}. More general statements appear in \cite[Thm.~5.6]{Brz:TSOC}, see also \cite[\S28.19]{BrzWis:CAC} in the setting of corings, or still more general using bimonads in  \cite[Thm.~6.11]{BruLacVir:HMOMC} or \cite[Thm.~5.3.1]{AguCha:GHMFB}.
\end{rem}

\section{Categorical equivalences}
\label{tetrayd}

\subsection{Hopf bimodules versus YD modules}
In this section, we would like to transfer the classical categorical equivalence between Hopf bimodules and Yetter-Drinfel'd modules for Hopf algebras that appeared in \cite[Cor.~6.4]{Schau:HMAYDM} to the realm of Hopf algebroids.

\subsubsection{The Yetter-Drinfel'd category and centre constructions}
%\hspace*{-.3cm}
A {\em left-right Yetter Drinfel'd} module $P$ over a left bialgebroid $(U,A)$ is
simultaneously a left $U$-module (with action denoted by juxtaposition) and a right $U$-comodule with coaction $\gr_P \colon P \to P \otimes_A \due U \lact {}, \ p \mapsto p_{[0]} \otimes_A p_{[1]}$ such that the two forgetful functors $\umod \to {}_\Ae \cM$ and $\comodu \to {}_\Ae \cM$ induce the same $A$-bimodule structure on $P$,
\begin{equation}
  \label{ydforget2}
a \lact p \ract a' = apa'
\end{equation}
for $a, a' \in A$ and $p \in P$,
and such that between $U$-action and $U$-coaction 
 the compatibility
\begin{equation}
  \label{yd2}
u_{(1)} p_{[0]} \otimes_A u_{(2)} p_{[1]} = (u_{(2)} p)_{[0]}  \otimes_A  (u_{(2)} p)_{[1]}u_{(1)}
\end{equation}
holds for all $u \in U$, $p \in P$. The corresponding category $\ydu$ is monoidal with respect to $\otimes_A$, unit object the base algebra $A$, left $U$-action given by the diagonal action, and right $U$-coaction using the codiagonal coaction, but with a flip; that is, for any
for $P,Q \in \ydu$, % and any $p \in P, q \in Q$,
one takes
\begin{equation}
  \label{carolafelsen}
P \otimes_A Q \to (P \otimes_A Q) \otimes_A U, \quad
p \otimes_A q \mapsto (p_{(0)} \otimes_A q_{(0)}) \otimes_A q_{(1)} p_{(1)}.
\end{equation}
As expected, this category is braided, with braiding
\begin{equation}
  \label{ydbraid}
\gs_{P,Q} \colon P \otimes_A Q \to Q \otimes_A P, \quad p \otimes_A q \mapsto q_{[0]} \otimes_A q_{[1]} p,
  \end{equation}
commonly known as {\em (left-right) Yetter-Drinfel'd braiding}. 
As this braiding for a general left bialgebroid is not necessarily invertible (but is so if the left bialgebroid is a left Hopf algebroid such that $Q$ becomes a left Hopf-Galois module, in which case the inverse is given by $q \otimes_A p \mapsto q_- p \otimes_A q_+$), we sometimes underline this fact by calling the braiding a {\em pre-}braiding.
Observe that
one can still define {\em left-left} YD modules, but in contrast to the case of bialgebras,
there are {\em no} right-left or right-right YD modules over left bialgebroids.

%it can be shown that it is (monoidally) equivalent to the (right weak) monoidal centre of the category $\umod$.

\begin{prop}
\label{endlichwaermerjetzt?}
  Let $(U,A)$ be a left bialgebroid.
  \begin{enumerate}
    \compactlist{50}
  \item
Then there is a monoidal equivalence of categories between $\ydu$ and the (right weak) centre $\cZ^r({}_U \cM)$ of the category of left $U$-modules.
  \item
    There is also a monoidal equivalence of categories between $\ydu$ and the (left weak) centre $\cZ^\ell(\cM^U)$ of the category of right $U$-comodules.
    \end{enumerate}   
\end{prop}

\begin{proof}
  \
  \begin{enumerate}
    \compactlist{50}
  \item
This has been proven in \cite[Prop.~4.4]{Schau:DADOQGHA} for the category of left-left Yetter-Drinfel'd modules, the proof of which can be easily adapted to the present situation, turning from the left weak centre to the right one.
  \item
The monoidal structure in $(\cM^U, \otimes_A, A)$ here is as in Eq.~\eqref{carolafelsen}, that is, with a flip; the bialgebroid $U$ itself with its coproduct is in $\cM^U$ if equipped with the $A$-bimodule structure $\due U \blact \ract$, that is, by means of the target maps.
    Let then $(P, \gs_{P,-})$ be an object in  $\cZ^\ell(\cM^U)$, which, in particular, means that for the {\em central structure} $\gs_{P,-}$ the hexagon axiom
    \begin{equation}
      \label{hexagon}
      \gs_{P, M \otimes_A N} =
      (M \otimes_A \gs_{P, N}) \circ (\gs_{P, M} \otimes_A N) 
    \end{equation}
    holds for all $M, N \in \cM^U$, suppressing all associators. We claim that 
    \begin{equation}
      \label{packnpost}
(\gve \otimes_A P) \circ \gs_{P,U} \colon P \otimes_A U \to P,
    \end{equation}
    despite the order in which it is written,    defines a {\em left} $U$-action on $P$, suppressing the canonical isomorphism $A \otimes_A P \simeq P$, and that with this left $U$-action $P \in \cM^U$ becomes a left-right YD module over $U$. To see this, let us show the associativity of \eqref{packnpost}: for $u, v \in U$ and $p \in P$, and denoting by $\mu^\op$ the opposite multiplication in $U$, we have
  \begin{eqnarray*}
(vu)p 
&
  {\overset{\scriptscriptstyle{
\eqref{packnpost}
}}{=}}
    &
\big((\gve \otimes_A P) \circ \gs_{P,U} \circ (P \otimes_A \mu^\op)\big)\big(p \otimes_A (u \otimes_A v) \big).
  \end{eqnarray*}
    From the naturality of $\gs$, we have 
    $
    \gs_{P,U} \circ (P \otimes_A \mu^\op) =
    (\mu^\op \otimes_A P) \circ \gs_{P,U \otimes_A U}
    $
since $\mu^\op$ is a morphism in $\cM^U$ with respect to the monoidal structure \eqref{carolafelsen}, and therefore
  \begin{eqnarray*}
(\gve \otimes_A P) \circ \gs_{P,U} \circ (P \otimes_A \mu^\op)
&
  {\overset{\scriptscriptstyle{
%\eqref{camelion}
}}{=}}
    &
(\gve \otimes_A P) \circ (\mu^\op \otimes_A P) \circ \gs_{P,U \otimes_A U}
  \\
&
  {\overset{\scriptscriptstyle{
\eqref{hexagon}
}}{=}}
    &
  ((\gve \circ \mu^\op) \otimes_A P) \circ
  (U \otimes_A \gs_{P, U}) \circ (\gs_{P, U} \otimes_A U). 
 \end{eqnarray*}
On the other hand, 
  \begin{eqnarray*}
v(up) 
&
  {\overset{\scriptscriptstyle{
\eqref{packnpost}
}}{=}}
    &
\big((\gve \otimes_A P) \circ \gs_{P,U}\big) \big(up \otimes_A v \big).
\\
&
  {\overset{\scriptscriptstyle{
\eqref{packnpost}
}}{=}}
    &
  \pig((\gve \otimes_A P) \circ \gs_{P,U} \circ \big((\gve \otimes_A P) \otimes_A U\big) \circ  (\gs_{P,U} \otimes_A U) \pig)
  \pig((p \otimes_A u) \otimes_A v \pig).
  \end{eqnarray*}
  Considering that $\gs_{P,U}$ is a morphism in $\cM^U$, therefore a morphism in ${}_\Ae \cM$, and hence, in particular,
  $\gs_{P,U} (ap \otimes_A u) = a \blact \gs_{P,U}(p \otimes_A u)$, where on the right hand side the left $A$-action is on the first tensor factor, we obtain:
 \begin{eqnarray*}
&&
   (\gve \otimes_A P) \circ \gs_{P,U} \circ \big((\gve \otimes_A P) \otimes_A U\big) \circ  (\gs_{P,U} \otimes_A U)
\\
&
  {\overset{\scriptscriptstyle{
    }}{=}}
&
  (\gve \otimes_A P)
  \circ (\mu^\op \otimes_A P) \circ
 (t \gve \otimes_A U \otimes_A P)
  \circ (U \otimes_A \gs_{P,U})
  \circ  (\gs_{P,U} \otimes_A U)
\\
&
  {\overset{\scriptscriptstyle{
    }}{=}}
  &
  ((\gve \circ \mu^\op) \otimes_A P)
\circ (U \otimes_A \gs_{P,U})
  \circ  (\gs_{P,U} \otimes_A U), 
 \end{eqnarray*}
 where the last step follows from $\gve(uv) = \gve(u t\gve(v))$ in a left bialgebroid, see \eqref{counityippieh}. By comparison, we observe that $(vu)p = v(up)$, and therefore \eqref{packnpost} defines a left $U$-action.

 Let us show next the defining identity \eqref{yd2} of a left-right YD module with respect to this left action on $(P, \gs_{P,-}) \in \cZ^\ell(\cM^U)$ and its natural right coaction. To this end, for a right comodule $M \in \cM^U$ with coaction $\rho_M$, consider first the map $M \otimes_A \gve \colon M \otimes_A \due U \lact {} \to M$, which is a right $U$-comodule map if $M \otimes_A \due U \lact {}$ is seen as a right comodule via the map $m \otimes_A u \mapsto (m_{[0]} \otimes_A 1_U) \otimes_A m_{[1]} \ract \gve(u)$. Furthermore, for every fixed $m \in M$, consider the map $L_m \colon U \to M \otimes_A \due U \lact {}, \ u \mapsto m \otimes_A u$, which is also a map of right comodules if $M \otimes_A \due U \lact {}$ is equipped with the right coaction $M \otimes_A \gD$. Using then the naturality of $\gs$ with respect to these two maps along with counitality, we have for any $M \in \cM^U$ and $m \in M, p \in P$:
 \begin{eqnarray*}
\gs_{P,M} (p \otimes_A m)
&
  {\overset{\scriptscriptstyle{
    }}{=}}
&
\gs_{P,M} \big(p \otimes_A m_{[0]} \gve(m_{[1]})\big)
\\
&
  {\overset{\scriptscriptstyle{
    }}{=}}
  &
  \pig(\gs_{P,M} \circ \big(P \otimes_A (M \otimes_A \gve)\big)\pig)\big(p \otimes_A (m_{[0]} \otimes_A m_{[1]}) \big)
\\
  &
  {\overset{\scriptscriptstyle{
    }}{=}}
  &
  \pig(\big((M \otimes_A \gve) \otimes_A P \big) \circ \gs_{P,M \otimes_A U}  \pig)\big(p \otimes_A (m_{[0]} \otimes_A m_{[1]}) \big)
\\
  &
  {\overset{\scriptscriptstyle{
    }}{=}}
  &
  \pig(\big((M \otimes_A \gve) \otimes_A P \big) \circ \gs_{P,M \otimes_A U}  \circ (P \otimes_A L_{m_{[0]}}) \pig)\big(p \otimes_A m_{[1]} \big)
\\
  &
  {\overset{\scriptscriptstyle{
    }}{=}}
  &
  \pig(\big((M \otimes_A \gve) \otimes_A P \big) \circ (L_{m_{[0]}} \otimes_A P) \circ \gs_{P, U}
  \pig)
  \big(p \otimes_A m_{[1]} \big)
\\
  &
  {\overset{\scriptscriptstyle{
    }}{=}}
  &
\big((M \otimes_A \gve) \otimes_A P \big)\big( m_{[0]}  \otimes_A \gs_{P, U}(p \otimes_A m_{[1]}) \big)
\\
  &
  {\overset{\scriptscriptstyle{
    }}{=}}
  &
 m_{[0]}  \otimes_A \big((\gve \otimes_A P ) \circ
  \gs_{P, U}\big)(p \otimes_A m_{[1]}) 
\\
  &
  {\overset{\scriptscriptstyle{
\eqref{packnpost}
    }}{=}}
  &
 m_{[0]}  \otimes_A m_{[1]} p.
 \end{eqnarray*}
 With this, and since $\gs_{P,U}$ is right $U$-colinear, one computes
 \begin{eqnarray*}
 u_{(1)} p_{[0]} \otimes_A u_{(2)} p_{[1]}
  &
  {\overset{\scriptscriptstyle{
\eqref{packnpost}
    }}{=}}
  &
\big((\gve \otimes_A P) \circ \gs_{P,U}\big)(p_{[0]} \otimes_A  u_{(1)})  \otimes_A u_{(2)} p_{[1]}
\\
  &
  {\overset{\scriptscriptstyle{
\eqref{carolafelsen}
    }}{=}}
  &
\big(\gve \otimes_A P \otimes_A U\big)\big(\gs_{P,U}(p \otimes_A  u)_{[0]}  \otimes_A \gs_{P,U}(p \otimes_A  u)_{[1]}\big)
\\
  &
  {\overset{\scriptscriptstyle{
    }}{=}}
  &
\big(\gve \otimes_A P \otimes_A U\big)\big((u_{(1)} \otimes_A u_{(2)}p)_{[0]}  \otimes_A (u_{(1)} \otimes_A u_{(2)}p)_{[1]}\big)
\\
  &
  {\overset{\scriptscriptstyle{
\eqref{carolafelsen}
    }}{=}}
  &
  \big(\gve \otimes_A P \otimes_A U\big)\big((u_{(1)} \otimes_A (u_{(3)}p)_{[0]})  \otimes_A (u_{(3)}p)_{[1]}u_{(2)}\big)
  \\
  &
  {\overset{\scriptscriptstyle{
\eqref{hak}
    }}{=}}
  &
 (u_{(2)} p)_{[0]}  \otimes_A  (u_{(2)} p)_{[1]}u_{(1)},
\end{eqnarray*}
 and hence $P$ fulfills the identity \eqref{yd2} and is therefore a left-right YD module, as desired.
 
The more obvious verification that one also obtains a functor $\ydu \to \cZ^\ell(\cM^U)$ is omitted.
    \end{enumerate}   
  \end{proof}

\subsubsection{Free Hopf (bi)modules}
%\hspace*{-.3cm}
Guided by the spirit that, according to the Fundamental Theorem \ref{spinoza}, one can see Hopf modules in a certain sense as {\em free}, let us consider Hopf modules of this particular type first:

\begin{lem}
  \label{aeun}
  Let $P \in \ydu$ be a left-right YD module over a left bialgebroid $(U,A)$. Then the tensor product $U_\ract \otimes_A P$ is a Hopf bimodule by means of the following structure maps:
  \begin{equation}
\label{mist}
\begin{array}{rcl}
  w(u \otimes_A p) &:=& w_{(1)} u \otimes_A w_{(2)} p,
  \\
  (u \otimes_A p)w &:=& uw \otimes_A p,
  \\
  \gl(u \otimes_A p) &:=& u_{(1)} \otimes_A (u_{(2)} \otimes_A p),
  \\
  \gr(u \otimes_A p) &:=& (u_{(1)} \otimes_A p_{[0]}) \otimes_A p_{[1]} u_{(2)},
  \end{array}
  \end{equation}
  for all $p \in P$ and $u, w \in U$.
 %%  In particular, 
%%   one has
%%   $$
%% {}^{\co\,\,U} (U_\ract \otimes_A P) \simeq P
%% $$
%% with respect to the left $U$-coaction.
  \end{lem}

\begin{proof}
First, it is obvious that left and right $U$-action commute, that is, $U_\ract \otimes_A P$ is a $U$-bimodule. By coassociativity of $U$, it is almost equally simple to see that the left and right coactions commute, that is, $U_\ract \otimes_A P$ is a $U$-bicomodule as well.

  Let us then discuss the corresponding four $A$-actions on $U_\ract \otimes_A P$ in the sense of Eqs.~\eqref{rain}:
the left $U$-action given in Eqs.~\eqref{mist} induces the $A$-bimodule structure
$$
a \lact (u \otimes_A p) \ract b := a \lact u \otimes_A p \ract b
= s(a) u \otimes_A t(b)p
$$
on $U_\ract \otimes_A P$, while 
the right one induces
$$
a \blact  (u \otimes_A p) \bract b := a \blact u \bract b \otimes_A p
= u t(a) s(b) \otimes_A p
.
$$
Likewise, the left $U$-coaction in \eqref{mist} induces the $A$-bimodule structure
  $$
a \cdot (u \otimes_A p) \cdot b := a \lact u \bract b \otimes_A p = s(a)us(b) \otimes_A p, 
$$
on $U_\ract \otimes_A P$,
while the right $U$-coaction induces
$$
a (u \otimes_A p) b := a \blact u  \otimes_A p b = ut(a) \otimes_A t(b)p. 
$$
From this it is clear that all conditions regarding the various $A$-actions in Definitions \ref{genzano} \& \ref{frascati}, respective parts (a) and (b), for the various kinds of Hopf modules are fulfilled.

 Finally,  it is a longish but straightforward check that $U_\ract \otimes_A P$ is a Hopf module in all four senses; to exemplify, we only check that it is an object in ${}_U \cM^U$, the other verifications being even simpler. One has
  \begin{equation*}
    \begin{split}
    \rho\big(w(u \otimes_A p)\big) &= \rho(w_{(1)} u \otimes_A w_{(2)} p)
    \\
    &=
    \big( w_{(1)} u_{(1)} \otimes_A (w_{(3)} p)_{[0]} \big) \otimes_A
    (w_{(3)} p)_{[1]}w_{(2)} u_{(2)}
\\
    &=
    ( w_{(1)} u_{(1)} \otimes_A w_{(2)} p_{[0]} ) \otimes_A
    w_{(3)} p_{[1]} u_{(2)}
    \\
    &=
    w \mancino \rho(u \otimes p),
\end{split}
    \end{equation*}
  which is condition \eqref{accedi1} in the compact formulation \eqref{sunshine1} of Proposition \ref{snow}, and 
  where in step three we used the defining left-right YD property \eqref{yd2}.
%
%  The last statement regarding the left coinvariants is obvious.
\end{proof}

Strengthening Theorem \ref{spinoza} by exploiting the adjoint action found in \cite[Lem.~2.7]{BekKowSar:UEAOLRACPCAC} on what was called a Hopf kernel there, we can then state:

\begin{theorem}
  \label{rhodia}
  Let $(U,A)$ be right Hopf bialgebroid (over a left bialgebroid) and such that $U_\ract$ is flat over $A$, and
  let $M\in {}_U^{\mkern 2mu U}\mkern -1mu\cM^U_U$ be a Hopf bimodule over $U$.
  \begin{enumerate}
    \compactlist{99}
    \item
  Then the module $ \coM$ of left coinvariants becomes a left
  $\Ae$-module by means of 
  \begin{equation}
    \label{wes1}
  \Ae \otimes \coM \to \coM, \quad (a,b) \otimes m \mapsto  a \blact m \ract b,
  \end{equation}
  which extends to a well-defined left $U$-action
  \begin{equation}
    \label{wes2}
\mpact \colon \due U \blact \bract \otimes_{\Ae} \coM \to \coM, \quad u \otimes_{\Ae} m \mapsto u_{\smap } m u_{\smam },
  \end{equation}
  termed the {\em adjoint} action.
\item
The two functors
\begin{equation}
  \label{funky}
  {}^{\co\,U } \! (-) \colon {}_U^{\mkern 2mu U}\mkern -1mu\cM^U_U
\quad \raisebox{-3pt}{$\longleftrightarrows$} \quad \ydu \ \colon \!  U_\ract \otimes_A - 
  \end{equation}
are mutually adjoint and  establish an equivalence of categories, with counit of the adjunction defined by the isomorphism 
  $$
\xi \colon  U_\ract  \otimes_A\coM \to  M, \quad u\otimes_A {m}\mapsto {m}u, 
$$
in the category ${}_U^{\mkern 2mu U}\mkern -1mu\cM^U_U$ of Hopf bimodules, and inverse given by
$$
%\xi^{-1} \colon
M \to U_\ract  \otimes_A \coM, \quad
m\mapsto m_{(-1)}\otimes_A m_{(0)\smap } m_{(0)\smam }.
$$
Its
unit, on the other hand, arises from the isomorphism
$$
\eta \colon P \to {}^{\co\,U} \! (U_\ract \otimes_A P),
\quad
p\mapsto 1_U \otimes_A p,
$$
in the category $\ydu$ of left-right YD modules over $U$,
and inverse
$
%\eta^{-1} \colon  {}^{\co\,U} (U\otimes_AV)
%\to {}_A\cM,
%\quad
u \otimes_A p \mapsto \varepsilon(u) \cdot p.
$
\end{enumerate}
\end{theorem}

\begin{proof}
  Let us first show that the map \eqref{wes2} is well-defined indeed over the Sweedler presentation \eqref{Tch1}: for $m \in \coM$ and $a \in A$, one has, using the right $A$-action on $M$, using~Eqs.~\eqref{rain},
  \begin{equation}
    \label{gaudia}
\gl_M(m\cdot a) = \gl_M(m \bract a) = m_{(-1)} s(a) \otimes_A m_{(0)} = s(a) \otimes_A m 
  \end{equation}
from the $A$-linearity of $\gl_M$; the last step follows from $m$ being coinvariant, {\em cf.}\ Eq.\ \eqref{waterman}. Applying $\varepsilon \otimes_A M$ to both sides and identifying $A \otimes_A M \cong M$ by means of $a \otimes_A m \mapsto a \cdot m$ yields $m \cdot a = a \cdot m$ and since $M \in {}^{U\!}_U\cM^U_U$, Eqs.~\eqref{rain} then imply
$$
m \cdot a = m \bract a = a \lact m = a \cdot m
$$
for
$m \in \coM$,
which gives the well-definedness of the map \eqref{wes2}. Note that \eqref{gaudia} also implies, in general, that $m \cdot a = a \cdot m \notin \coM$ for $m \in \coM$ and all $a \in A$. Next, we have to show that the action map \eqref{wes2} lands in $\coM$: indeed, for $m \in \coM$ and $u \in U$, using the fact that $M$ is a Hopf bimodule over $U$, that is, in particular a left-right and left-left Hopf module, we have
\begin{equation*}
\begin{split}
 \gl_M(u \pmact m) =
  \gl_M(u_{\smap } m u_{\smam }) &= u_{\smap (1)} m_{(-1)} u_{\smam (1)} \otimes_A u_{\smap (2)} m_{(0)} u_{\smam (2)}
  \\
&  = u_{\smap (1)} u_{\smam (1)} \otimes_A u_{\smap (2)} m u_{\smam (2)}
  \\
&  = u_{\smap \smap (1)} u_{\smap \smam } \otimes_A u_{\smap \smap (2)} m u_{\smam }
  \\
&  = 1 \otimes_A u_{\smap } m u_{\smam } = 1 \otimes_A u \pmact m,
\end{split}
  \end{equation*}
using the identities \eqref{Tch5} resp.\ \eqref{Tch2} in the fourth resp.\ fifth step.
That the so-defined map \eqref{wes2} then defines a left $U$-action is immediate from \eqref{Tch6}, while the statement regarding the $A$-bimodule structure follows from \eqref{Tch9}.

Next, let us prove that the functors \eqref{funky} do what they are supposed to do, {\em i.e.}, first, 
that the functor
  ${}^{\co\,U} \! (-)$
of taking left coinvariants does produce a (left-right) YD module when starting from a Hopf bimodule over $U$, and, vice versa, that the functor $U_\ract \otimes -$ of tensoring with $U$ produces a Hopf bimodule when starting from a (left-right) YD module.

Let $M \in \tetra$. As seen in part (i), the space of coinvariants
$\coM$ is a left $\Ae$-module resp.\
a left
$U$-module with respect to the actions \eqref{wes1} resp.\ \eqref{wes2}, and we claim that, via the coaction $\rho_M$ inherited from $M$, it is also a right $U$-comodule that turns it into a left-right YD module. The first claim follows from the fact that $M$ is, in particular, a $U$-bicomodule, that is,
$(U \otimes_A \rho_M) \circ \gl_M = (\gl_M \otimes_A U) \circ \rho_M$, which means that $\rho_M$ restricts to a well-defined map $\coM \to \coM \otimes_A U$.
Let us then show that $\rho_M$ and the left action \eqref{wes2} endow $\coM$ with a left-right YD module structure. For $m \in \coM$, one has
  \begin{eqnarray*}
&&
    (u_{(2)} \pmact m)_{[0]}  \otimes_A (u_{(2)} \pmact m)_{[1]}u_{(1)}
\\
&
  {\overset{\scriptscriptstyle{
\eqref{wes2}
}}{=}}
    &
  (u_{(2)\smap} m u_{(2)\smam})_{[0]}  \otimes_A (u_{(2)\smap} m u_{(2)\smam})_{[1]}u_{(1)}
\\
&
  {\overset{\scriptscriptstyle{
\eqref{accedi1}, \eqref{accedi2}
}}{=}}
    &
  u_{(2)\smap(1)} m_{[0]} u_{(2)\smam (1)}  \otimes_A 
u_{(2)\smap(2)} m_{[1]} u_{(2)\smam (2)}
u_{(1)}
\\
&
  {\overset{\scriptscriptstyle{
\eqref{Tch5}
}}{=}}
    &
  u_{(2)\smap\smap(1)} m_{[0]} u_{(2)\smap\smam}  \otimes_A 
u_{(2)\smap\smap(2)} m_{[1]} u_{(2)\smam}
u_{(1)}
\\
&
  {\overset{\scriptscriptstyle{
\eqref{Tch3}
}}{=}}
    &
  u_{\smap(1)} m_{[0]} u_{\smam}  \otimes_A 
  u_{\smap(2)} m_{[1]}
\\
&
  {\overset{\scriptscriptstyle{
\eqref{Tch4}, \eqref{wes2}
}}{=}}
    &
  u_{(1)} \mpact m_{[0]} \otimes_A 
  u_{(2)} m_{[1]},
  \end{eqnarray*}
  which is the left-right YD property as given in \eqref{yd2}. That the functor $U_\ract \otimes_A -$ does produce a Hopf bimodule over $U$ out of a left-right YD module is the content of Lemma \ref{aeun}. Finally, the second part of part (ii) follows from Theorem \ref{spinoza}, once we add that $\xi$ is, apart from being a morphism of right $U$-modules and left $U$-comodules as in Theorem \ref{spinoza}, also a morphism of left $U$-modules as well as of right $U$-comodules (and likewise for $\eta$). The left $U$-linearity of $\xi$ follows from
$$
\xi\big(u(u' \otimes_A m)\big) = 
\xi\big(u_{(1)}u' \otimes_A u_{(2)\smap } m u_{(2)\smam }\big)
=  u_{(2)\smap } m u_{(2)\smam } u_{(1)}u'
= umu' = u\xi(u' \otimes_A m),
$$
where we used \eqref{Tch3} in the third step, while right $U$-colinearity of $\xi$ simply follows from the fact that, by Lemma \ref{aeun}, the tensor product $U_\ract \otimes_A P$ is also a right-right Hopf module, {\em cf.}\ Eq.~\eqref{sunshine1}. The corresponding statement about $\eta$ is obvious.
\end{proof}

\subsection{Monoidality and braidings}
As a next goal, we want to promote Theorem \ref{rhodia} to an equivalence of {\em monoidal} categories, which, in turn, can be enhanced even to an equivalence of {\em braided} monoidal categories.
To start with, 
on the one hand, the monoidal structure on $\ydu$ is well-known and given simply by the tensor product $\otimes_A$, diagonal action and codiagonal coaction, with unit object given by the base algebra $A$: in short, we shall write this monoidal category as 
$\big(\ydu, \otimes_A, A\big)$.
On the other hand, let us discuss two possible monoidal structures on $\tetra$, which appear in a natural way: 

\begin{lem}
\label{exposition}
  Let $U$ be a left bialgebroid and let $\tetra$ denote the corresponding category of Hopf bimodules over $U$. Then
  \begin{enumerate}
    \compactlist{50}
  \item
the triple $\big(\tetra, \otimes_{\mkern 1mu U} , U\big)$ forms a monoidal category;
  \item
    the triple $\big(\tetra, \bx_U, U\big)$ forms a (lax) monoidal category if both $U_\ract$ and $\due U \lact {}$ are flat over~$A$.
  \end{enumerate}
    \end{lem}

\begin{proof}
Observe that in both cases the unit object is meant to be the (total space of the) left bialgebroid $U$ itself.
\begin{enumerate}
  \compactlist{99}
  \item
    Let $M, N \in \tetra$. The monoidal structure arises from
the following (left and right) $U$-actions and $U$-coactions 
on $M \otimes_{\mkern 1mu U}  N$:
      \begin{equation}
\label{kran1}
\begin{array}{rcl}
  u(m \otimes_{\mkern 1mu U}  n) &:=& um \otimes_{\mkern 1mu U}  n,
  \\
  (m \otimes_{\mkern 1mu U}  n)u &:=& m \otimes_{\mkern 1mu U}  nu,
  \\
  \gl(m \otimes_{\mkern 1mu U}  n) &:=& m_{(-1)}n_{(-1)} \otimes_A (m_{(0)} \otimes_{\mkern 1mu U}  n_{(0)}),
  \\
  \gr(m \otimes_{\mkern 1mu U}  n) &:=& (m_{[0]} \otimes_{\mkern 1mu U}  n_{[0]}) \otimes_A m_{[1]} n_{[1]},
  \end{array}
  \end{equation}
      for $u \in U$, $m \in M$, and $n \in N$.
These induce the following respective $A$-bimodule structures: 
       \begin{equation*}
%\label{kran2}
\begin{array}{rclcl}
  a \lact (m \otimes_{\mkern 1mu U}  n) \ract b &:=& (a \lact m \ract b) \otimes_{\mkern 1mu U}  n &=& s(a)t(b) m \otimes_{\mkern 1mu U}  n,
  \\
  a \blact (m \otimes_{\mkern 1mu U}  n) \bract b &:=& m \otimes_{\mkern 1mu U}
  (a \blact n \bract b) &=& m \otimes_{\mkern 1mu U}  n t(a) s(b),
  \\
  a \cdot (m \otimes_{\mkern 1mu U}  n) \cdot b &:=&
  (a \cdot m) \otimes_{\mkern 1mu U}  (n \cdot b) &=& s(a) m \otimes_{\mkern 1mu U}  n s(b), 
  \\
  a(m \otimes_{\mkern 1mu U}  n)b &:=& (mb) \otimes_{\mkern 1mu U}  (an) &=& t(b)m \otimes_{\mkern 1mu U}  nt(a), 
  \end{array}
  \end{equation*}
    from which it is obvious that the four requirements in Eqs.~\eqref{rain} are fulfilled. 
Checking that the left resp.\ right coaction in \eqref{kran1} are both left and right $U$-linear with respect to the given actions so that by \eqref{sunshine1}--\eqref{sunshine2} one obtains a Hopf bimodule structure on $M \otimes_{\mkern 1mu U}  N$ is a simple exercise (hence omitted) that uses the Hopf bimodule properties of $M$ and $N$ along with the diagonal action from left and right on the codomain of the respective coproducts as in \eqref{fog}. Clearly, $U$ is the unit object in this category by means of its left and right multiplication along with its coproduct.
\item
  Recall first that the cotensor product over $U$ is defined by
  \begin{equation}
    \label{burgol}
M \bx_U  N := \big\{m \otimes_A n \in M_\ract \otimes_A \due N \lact {} \mid m_{[0]} \otimes_A m_{[1]} \otimes_A n = m \otimes_A n_{(-1)} \otimes_A n_{(0)}  \big\},
  \end{equation}
  that is, the equaliser of the pair of maps $\rho_M \otimes_A N$  and $M \otimes_A \gl_N$.

  Let then $M, N \in \tetra$. The monoidal structure arises from
the following (left and right) $U$-actions and $U$-coactions 
on $M \bx_U  N$:
      \begin{equation}
\label{kran2}
\begin{array}{rcl}
  u(m \bx_U  n) &:=& u_{(1)}m \bx_U  u_{(2)}n,
  \\
  (m \bx_U  n)u &:=& mu_{(1)} \bx_U  nu_{(2)},
  \\
  \gl(m \bx_U  n) &:=& m_{(-1)} \otimes_A (m_{(0)} \bx_U  n),
  \\
  \gr(m \bx_U  n) &:=& (m \bx_U  n_{[0]}) \otimes_A n_{[1]},
  \end{array}
  \end{equation}
      for $u \in U$, $m \in M$, and $n \in N$.
      While it is well known that the cotensor product of bi\-co\-mod\-ules becomes a left resp.\ right $U$-comodule again under the stated left resp.\ right $A$-flatness conditions that guarantees associativity between the relevant tensor and cotensor product,
      %(and also associativity of the cotensor product itself; see, for example, \cite[\S22.3]{BrzWis:CAC}),
      let us examine the possibly unexpected fact  that the given right $U$-action is well-defined (see the comment in \S\ref{casadante} why this might come as a surprise): for this, we have to prove that $M \bx_U N \subseteq M \times_A N$, where the Takeuchi-Sweedler product is given as
\begin{equation*}
M \times_A N = \big\{ \textstyle\sum_i m^i \otimes n^i \in M \otimes_A N  \mid  \sum_i am^i  \otimes n^i = \sum_i m^i \otimes n^i \cdot a, \ \forall a \in A  \big\},
\end{equation*}
where, however, due to the Hopf bimodule structure on $M$ and $N$, the involved $A$-actions are given, in turn, by $am = a \blact m = mt(a)$ for $m \in m$, and $n \cdot a = n \bract a = ns(a)$, see Eqs.~\eqref{rain}.
Indeed, if $ m \bx_U n \in M \bx_U N$, then by \eqref{ha3}, \eqref{accedi2}, \eqref{hak}, and \eqref{rain}, this means
\begin{equation*}
  \begin{split}
(am)_{[0]} \otimes_A (am)_{[1]} \otimes_A n
    &=
    m_{[0]} \otimes_A a \blact m_{[1]} \otimes_A n
    \\
    &=
    m \otimes_A a \blact n_{(-1)} \otimes_A  n_{(0)} 
    \\
    &=
    m \otimes_A n_{(-1)} \otimes_A  n_{(0)} \cdot a.
  \end{split}
    \end{equation*}
Since the natural and the induced $A$-action on a comodule commute, applying $\gve$ to the middle factor of both the left resp.\ right hand side of the above yields the identity
$$
am \bx_U n = m \bx_U n \cdot a
$$
on the subspace $M \bx_U N$,
as desired. Similarly as before, the (co)actions in \eqref{kran2} then induce the following respective $A$-bimodule structures: 
        \begin{equation*}
\begin{array}{rclcl}
  a \lact (m \bx_U  n) \ract b &:=& (a \lact m)  \bx_U  (n \ract b) &=& s(a) m \bx_U  t(b)n,
  \\
 a \blact (m \bx_U  n) \bract b &:=& (m \bract b) \bx_U  (a \blact n)  &=& m s(b) \bx_U  n t(a),
  \\
  a \cdot (m \bx_U  n) \cdot b &:=&
  m \bx_U  (a \cdot n \cdot b) &=& m \bx_U s(a) n s(b), 
  \\
  a(m \bx_U  n)b &:=& (amb) \bx_U  n &=& t(b)mt(a) \bx_U  n, 
  \end{array}
  \end{equation*}
    from which, again, it is obvious that the four requirements in Eqs.~\eqref{rain} are fulfilled. 
Also, verifying that the left resp.\ right $U$-coaction in \eqref{kran2} are both left and right $U$-linear with respect to the given $U$-actions is a direct check so that by \eqref{sunshine1}--\eqref{sunshine2} one obtains a Hopf bimodule structure on $M \bx_U  N$ is straightforward, while again $U$ is the unit object in this category by means of its left and right multiplication along with its coproduct.
\end{enumerate}
This completes the proof.
\end{proof}

\begin{prop}
\label{resistente}
  If $(U,A)$ is a left Hopf algebroid (over a left bialgebroid) and such that both $U_\ract$ and $\due U \lact {}$ are flat over $A$, then
by means of the natural isomorphisms
  \begin{equation}
    \label{vipitenojoghurt}
\xi_{M,N} \colon M \otimes_{\mkern 1mu U} N \to M \bx_U N, \qquad m \otimes_U n \mapsto m_{[0]} n_{(-1)} \bx_U m_{[1]} n_{(0)}. 
  \end{equation}
  of which the inverses are given by
  $
\xi^{-1}_{M,N} \colon m \bx_U n \mapsto m_+m_- \otimes_U n,  
  $
the identity functor establishes an equivalence
\begin{equation}
  \label{cuffie}
\big(\tetra, \otimes_{\mkern 1mu U} , U\big) \to \big(\tetra, \bx_U , U\big)
 \end{equation}
of monoidal categories.
\end{prop}

 \begin{proof}
Let us first explain why the map \eqref{vipitenojoghurt} is well-defined over the two different Sweedler presentations (in spite of the fact that it does not look so at a first glance): well-definedness over the right coaction on $M$
in the tensor product $M \otimes_A N$
follows from the defining property of commuting right actions over $U$ resp.\ over $A$ in Definition \ref{genzano}, part (ii)\,(b), of a right-right Hopf module, while the well-definedness over the left coaction on $N$ still follows by Definition \ref{genzano}, part (ii)\,(a), along with the Takeuchi property \eqref{ha3} for a right $U$-comodule. Moreover, that the maps $\xi = \xi_{M,N}$ are $U$-balanced simply follows by \eqref{accedi1} and \eqref{accedi4},  as does the fact that it corestricts to $M \bx_U N$.

Let us then show that \eqref{vipitenojoghurt} is an isomorphism for all pairs of $M, N \in \tetra$: dropping the subscripts for the sake of readability, one has
  \begin{eqnarray*}
    (\xi^{-1} \circ \xi)(m \otimes_{\mkern 1mu U}  n)
&
   {\overset{\scriptscriptstyle{
\eqref{vipitenojoghurt}
}}{=}}
& 
   \xi^{-1}(m_{[0]} n_{(-1)} \bx_U m_{[1]} n_{(0)})
   \\
   &
   {\overset{\scriptscriptstyle{
%\eqref{vipitenojoghurt}
}}{=}}
& 
 (m_{[0]} n_{(-1)})_+(m_{[0]} n_{(-1)})_- \otimes_{\mkern 1mu U} m_{[1]} n_{(0)}
   \\
   &
   {\overset{\scriptscriptstyle{
\eqref{scalea2}
}}{=}}
& 
 m_{[0]+} n_{(-1)+} n_{(-1)-} m_{[0]-} \otimes_{\mkern 1mu U} m_{[1]} n_{(0)}
 \\
 &
 {\overset{\scriptscriptstyle{
\eqref{Sch7}, \eqref{Nch5}
}}{=}}
& 
 \big(\gve(n_{(-1)}) m_{[0]}\big)_+ \big(\gve(n_{(-1)}) m_{[0]}\big)_-
 \otimes_{\mkern 1mu U} m_{[1]} n_{(0)}
 \\
 &
 {\overset{\scriptscriptstyle{
\eqref{ha3}, \eqref{rain}
}}{=}}
& 
 m_{[0]+} m_{[0]-} \otimes_{\mkern 1mu U} m_{[1]} (\gve(n_{(-1)}) \lact n_{(0)})
\\
 &
 {\overset{\scriptscriptstyle{
%\eqref{ha3}, \eqref{rain}
}}{=}}
& 
 m_{[0]+} m_{[0]-}  m_{[1]} \otimes_{\mkern 1mu U} n
\\
 &
 {\overset{\scriptscriptstyle{
\eqref{Nch3}
}}{=}}
& 
 m \otimes_{\mkern 1mu U} n,
  \end{eqnarray*}
for $m \in M$, $n \in N$,  and hence $\xi^{-1} \circ \xi = \id$.
On the other hand, 
  \begin{eqnarray*}
    (\xi \circ \xi^{-1})(m \bx_U  n)
&
   {\overset{\scriptscriptstyle{
%\eqref{vipitenojoghurt}
}}{=}}
& 
 \xi(m_+ m_-   \otimes_{\mkern 1mu U} n )
   \\
   &
   {\overset{\scriptscriptstyle{
\eqref{vipitenojoghurt}
}}{=}}
& 
 (m_+m_-)_{[0]} n_{(-1)} \bx_U (m_+m_-)_{[1]} n_{(0)}
   \\
   &
   {\overset{\scriptscriptstyle{
\eqref{accedi2}
}}{=}}
& 
 m_{+[0]} m_{-(1)} n_{(-1)} \bx_U m_{+[1]} m_{-(2)} n_{(0)}
 \\
 &
 {\overset{\scriptscriptstyle{
\eqref{Nch4}
}}{=}}
& 
m_{++[0]} m_{-} n_{(-1)} \bx_U m_{++[1]} m_{+-} n_{(0)}
 \\
 &
 {\overset{\scriptscriptstyle{
\eqref{Nch2}
}}{=}}
& 
m_{+} m_{-} n_{(-1)} \bx_U n_{(0)}
\\
 &
 {\overset{\scriptscriptstyle{
\eqref{burgol}
}}{=}}
& 
m_{[0]+} m_{[0]-} m_{[1]} \bx_U n
\\
 &
 {\overset{\scriptscriptstyle{
\eqref{Nch3}
}}{=}}
 &
 m \bx_U n,
  \end{eqnarray*}
  and therefore $\xi \circ \xi^{-1} = \id$ as well.

  While the appurtenant map $\xi_0 \colon U \to U$ on the respective unit objects is clearly the identity on $U$, it is now an easy check that the identity functor along with $\xi$ resp.\ $\xi_0$ fulfil the hexagon resp.\ square axioms for monoidal functors (see, for example, \cite[Def.~XI.4.1]{Kas:QG} for details and terminology) and therefore establishes a monoidal equivalence.
  \end{proof}
  
With this, we can state:

\begin{theorem}
  \label{spielzeit}
  Let $(U,A)$ be both a left and right Hopf algebroid (over a left bialgebroid).
  \begin{enumerate}
    \compactlist{99}
  \item
      If $U_\ract$ is right $A$-flat, then the functors from \eqref{funky} induce a monoidal equivalence between
      $\big(\tetra, \otimes_{\mkern 1mu U} , U\big)$
      and
      $\big(\ydu, \otimes_A, A\big)$.
    \item
  If both $U_\ract$ and $\due U \lact {}$ are flat over $A$,
      then there is a monoidal equivalence between
      $\big(\tetra, \bx_U, U\big)$
      and
      $\big(\ydu, \otimes_A, A\big)$.
  \end{enumerate}
\end{theorem}

  \begin{proof}
    \
 \begin{enumerate}
    \compactlist{99}
  \item
    Thanks to Theorem \ref{rhodia}, it is sufficient  to check the statement for Hopf bimodules of the form $U \otimes_A P$ for $P \in \ydu$, and with Hopf bimodule structure as discussed in Lemma \ref{aeun}.
    We claim that
    \begin{equation}
      \label{rosenthalerplatz}
      \begin{array}{rcl}
            \zeta_{U \otimes_AP, U \otimes_A Q} \colon \quad    (U \otimes_A P) \otimes_{\mkern 1mu U}  (U \otimes_A Q) &\to& U \otimes_A (P \otimes_A Q),
        \\[1mm]
        (u \otimes_A p) \otimes_{\mkern 1mu U}  (v \otimes_A q) &\mapsto&
        u_{(1)}v \otimes_A (p_{[0]}  \otimes_A p_{[1]} u_{(2)} q),
\\[1mm]
  (1_U \otimes_A p_+) \otimes_{\mkern 1mu U}  (u \otimes_A p_-q) &\mapsfrom& u \otimes_A (p  \otimes_A q),
      \end{array}
      \end{equation}
    where $P, Q \in \ydu$, and where $P$ is seen as a {\em left} Hopf-Galois module in the sense of \S\ref{lostbuddha}, is an isomorphism of Hopf bimodules, {\em i.e.}, of $\Ae$-bimodules which, at the same time, is left and right $U$-linear and left and right $U$-colinear. If we call $\zeta :=             \zeta_{U \otimes_AP, U \otimes_A Q}$ and $\zeta^{-1}$ the map in the last line of \eqref{rosenthalerplatz}, then showing that $\zeta \circ \zeta^{-1}$ yields the identity on $U \otimes_A (P \otimes_A Q)$ is quickly seen by Eq.~\eqref{Nch2}; on the other hand, we compute:
    \begin{equation*}
      \begin{split}
         (\zeta^{-1} \circ \zeta)\big((u \otimes_A p) \otimes_{\mkern 1mu U}  (v \otimes_A q)\big)
                &=
        \zeta^{-1}\big(u_{(1)}v \otimes_A p_{[0]}  \otimes_A p_{[1]} u_{(2)} q\big)
        \\
        &
        =
        (1_U \otimes_A  p_{[0]+}) \otimes_{\mkern 1mu U}   (u_{(1)}v \otimes_A p_{[0]-} p_{[1]} u_{(2)} q)
        \\
        &
        =
        (1_U \otimes_A  p) \otimes_{\mkern 1mu U}   u(v \otimes_A q) =  (u \otimes_A  p) \otimes_{\mkern 1mu U}   (v \otimes_A q),
      \end{split}
      \end{equation*}
    using \eqref{Nch3} in the third step and the left and right $U$-action from Eqs.~\eqref{mist} on $U \otimes_A P$ in the last. Next, with respect to the right actions from \eqref{kran1} resp.\ \eqref{mist}, it is obvious that $\zeta$ is a morphism of right $U$-modules, while for the left $U$-actions, we have for $w \in U$,
\begin{equation*}
      \begin{split}
        \zeta \pig(w \big((u \otimes_A p) \otimes_{\mkern 1mu U}  (v \otimes_A q)\big) \pig)
        & =
        \zeta \big((w_{(1)}u \otimes_A w_{(2)}p) \otimes_{\mkern 1mu U}  (v \otimes_A q)\big)
        \\
        &=
        w_{(1)}u_{(1)}v \otimes_A (w_{(3)}p)_{[0]} \otimes_A (w_{(3)}p)_{[1]} w_{(2)} u_{(2)} q
         \\
        &=
        w_{(1)}u_{(1)}v \otimes_A w_{(2)}p_{[0]} \otimes_A w_{(3)}p_{[1]} u_{(2)} q
         \\
        &=
        w \zeta\big((u \otimes_A p) \otimes_{\mkern 1mu U}  (v \otimes_A q)\big),
      \end{split}
      \end{equation*}
using \eqref{mist} and \eqref{kran1} here repeatedly, while the YD property \eqref{yd2} was used in step three; hence, $\zeta$ is left $U$-linear.
As for colinearity, it is obvious, in turn, that $\zeta$ is left $U$-colinear, while right $U$-colinearity results from:
\begin{equation*}
      \begin{split}
&        (\rho \circ        \zeta) \big((u \otimes_A p) \otimes_{\mkern 1mu U}  (v \otimes_A q)\big)
        \\
        &
        =
        \rho \big( u_{(1)}v \otimes_A (p_{[0]}  \otimes_A p_{[1]} u_{(2)} q) \big)
        \\
&= \pig( u_{(1)}v_{(1)} \otimes_A \big(p_{[0]} \otimes_A  (p_{[2]} u_{(3)} q)_{[0]}\big)\pig) \otimes_A (p_{[2]} u_{(3)} q)_{[1]}p_{[1]} u_{(2)}v_{(2)}
       \\
       &= \big( u_{(1)}v_{(1)} \otimes_A (p_{[0]} \otimes_A  p_{[1]} u_{(2)} q_{[0]})\big) \otimes_A p_{[2]} u_{(3)} q_{[1]} v_{(2)}
       \\
&= \zeta\big((u_{(1)} \otimes_A p_{[0]}) \otimes_{\mkern 1mu U}   (v_{(1)} \otimes_A q_{[0]})\big) \otimes_A p_{[1]} u_{(2)} q_{[1]} v_{(2)}
       \\
       &=
   \big((\zeta \otimes_A U) \circ \rho\big)\big((u \otimes_A p) \otimes_{\mkern 1mu U}  (v \otimes_A q)\big) 
       ,
      \end{split}
      \end{equation*}
where we used \eqref{mist} in step two, the YD property \eqref{yd2} and the diagonal right coaction on the tensor product of YD modules in step three, and again \eqref{mist} along with \eqref{kran1} in the last step. In order to obtain a monoidal equivalence, verifying the hexagon (and square) axioms for the isomorphisms $\zeta_{U\otimes_A P, U \otimes_A Q}$ and the functor $U \otimes_A -$ is then a direct check, and therefore also for their inverses and the functor  ${}^{\co\,U } \! (-)$, which proves the monoidality of the two functors.
    \item
This now simply follows by the first part combined with Proposition \ref{resistente}.
 \end{enumerate}
 This concludes the proof.
 \end{proof}

  Finally, let us show that the monoidal equivalences from Theorem \ref{spielzeit} are actually equivalences of {\em (pre-)braided} monoidal categories, where $\ydu$, as mentioned before, is braided by the natural transformation in Eq.~\eqref{ydbraid}.
%
%% \begin{equation}
%%   \label{ydbraid}
%% \gs_{P,Q} \colon P \otimes_A Q \to Q \otimes_A P, \quad p \otimes_A q \mapsto q_{[0]} \otimes_A q_{[1]} p,
%%   \end{equation}
%% commonly known as {\em (left-right) Yetter-Drinfel'd braiding}. 
%% As this braiding for a general left bialgebroid is not necessarily invertible (but is so if the left bialgebroid is a left Hopf algebroid and $U_\ract$ right $A$-flat such that $Q$ becomes a left Hopf-Galois module, in which case the inverse is given by $q \otimes_A p \mapsto q_- p \otimes_A q_+$), we sometimes underline this fact by calling the braiding a {\em pre-}braiding.
%
Generalising the respective statement \cite[Thm.~6.3]{Schau:HMAYDM} for Hopf algebras, we can prove, by simultaneously assuming the existence of both a left and right Hopf structure on the same left bialgebroid in the sense of \S\ref{sokc}, that is, a left bialgebroid equipped with two different and invertible Hopf-Galois maps:

\begin{theorem}
  \label{buddha}
  Let $(U,A)$ be simultaneously a left and right Hopf algebroid (over a left bialgebroid) such that $U_\ract$ is flat over $A$.
  Then the monoidal category 
$\big(\tetra, \otimes_{\mkern 1mu U} , U\big)$ is braided with braiding
\begin{equation}
  \label{hopfbraid}
  \tau_{M,N} \colon M \otimes_{\mkern 1mu U}  N \to N \otimes_{\mkern 1mu U}  M, \quad m \otimes_{\mkern 1mu U}  n
  \mapsto
  m_{[1]\smap } n_{\smap } n_{\smam  +} \otimes_{\mkern 1mu U}    m_{[1]\smam  } m_{[0]} n_{\smam  -} 
,
\end{equation}
and the equivalence
between
      $\big(\tetra, \otimes_{\mkern 1mu U}, U\big)$
      and
      $\big(\ydu, \otimes_A, A\big)$
from Theorem \ref{spielzeit} becomes one of braided monoidal categories. 
\end{theorem}

\begin{proof}
  The fact that the map \eqref{hopfbraid} is well-defined over the {\em four} involved Sweedler presentations follows, after a little thought, the standard way using
  \eqref{rain}, \eqref{Sch9}, \eqref{Mch1}, \eqref{Tch9}, and \eqref{Tch6}, so that we skip its details. To show that it is balanced over the tensor product over $U$, we need the fact that $U$ is both left and right Hopf and hence the mixed Equations \eqref{mampf1}--\eqref{mampf3} hold. Indeed, on the one hand, 
  \begin{eqnarray*}
    \tau_{M,N} \colon mu \otimes_{\mkern 1mu U}  n
 & \mapsto &
  (mu)_{[1]\smap } n_{\smap } n_{\smam  +} \otimes_{\mkern 1mu U}    (mu)_{[1]\smam  } (mu)_{[0]} n_{\smam  -}
\\
&
   {\overset{\scriptscriptstyle{
\eqref{accedi2}
}}{=}}
& 
  (m_{[1]}u_{(2)})_{\smap } n_{\smap } n_{\smam  +} \otimes_{\mkern 1mu U}   (m_{[1]}u_{(2)})_{\smam  }
m_{[0]}u_{(1)} n_{\smam  -}
\\
&
 {\overset{\scriptscriptstyle{
\eqref{scalea}
}}{=}}
  &
  m_{[1]\smap }u_{(2)\smap } n_{\smap } n_{\smam  +} \otimes_{\mkern 1mu U}   u_{(2)\smam  } m_{[1]\smam  }
  m_{[0]}u_{(1)} n_{\smam  -}
    \end{eqnarray*}
hold, while, on the other hand,
  \begin{eqnarray*}
    \tau_{M,N} \colon m \otimes_{\mkern 1mu U}  un
 & \mapsto &
  m_{[1]\smap } (un)_{\smap } (un)_{\smam  +} \otimes_{\mkern 1mu U}    m_{[1]\smam  } m_{[0]} (un)_{\smam  -}
\\
&
  {\overset{\scriptscriptstyle{
\eqref{scalea}, \eqref{scalea2}
}}{=}}
& 
  m_{[1]\smap } u_{\smap } n_{\smap } n_{\smam  +} u_{\smam  +} \otimes_{\mkern 1mu U}    m_{[1]\smam  } m_{[0]} u_{\smam  -} n_{\smam  -}
  \\
&
  {\overset{\scriptscriptstyle{
\eqref{mampf3}
}}{=}}
  &
  m_{[1]\smap }u_{(2)\smap } n_{\smap } n_{\smam  +} \otimes_{\mkern 1mu U}   u_{(2)\smam  } m_{[1]\smam  }
  m_{[0]}u_{(1)} n_{\smam  -}
    \end{eqnarray*}
is true,  using the tensor product over $U$ in the last step. Hence, the map $    \tau_{M,N}$ is $U$-balanced.

  Following the idea in the proof of \cite[Thm.~6.3]{Schau:HMAYDM},
by Theorem \ref{rhodia} and its variant for the left-left case the starting point of which has been mentioned in Remark \ref{nunzia}, it is enough to consider two Hopf modules of the form $U \otimes_A P$ and $U \otimes_A Q$ with $P, Q \in \ydu$ as in Lemma~\ref{aeun}. Observe that
  \begin{eqnarray*}
&&
    \big( \zeta^{-1} \circ (U \otimes_A \gs_{P,Q}) \circ \zeta \big)\big((u \otimes_A p) \otimes_{\mkern 1mu U}  (v \otimes_A q)\big)
\\
&
  {\overset{\scriptscriptstyle{
%\eqref{wes2}
}}{=}}
&
  \big( \zeta^{-1} \circ (U \otimes_A \gs_{P,Q}) \big)
      \big(u_{(1)}v \otimes_A (p_{[0]}  \otimes_A p_{[1]} u_{(2)} q)\big)
        \\
        &
  {\overset{\scriptscriptstyle{
%\eqref{wes2}
}}{=}}
&
  \zeta^{-1}\pig(
u_{(1)}v \otimes_A
\big((p_{[1]} u_{(2)} q)_{[0]} \otimes_A 
(p_{[1]} u_{(2)} q)_{[1]}  p_{[0]}  \big)
  \pig)
        \\
       &
  {\overset{\scriptscriptstyle{
%\eqref{wes2}
}}{=}}
&
        \big(1_U \otimes_A (p_{[1]} u_{(2)} q)_{[0]+} \big)
\otimes_{\mkern 1mu U}
\big(u_{(1)}v \otimes_A
(p_{[1]} u_{(2)} q)_{[0]-} (p_{[1]} u_{(2)} q)_{[1]}  p_{[0]}  \big)
\\
    &
  {\overset{\scriptscriptstyle{
\eqref{Nch3}
}}{=}}
&
        (1_U \otimes_A p_{[1]} u_{(2)} q )
\otimes_{\mkern 1mu U}
(u_{(1)}v \otimes_A
p_{[0]})
       \\
        &
  {\overset{\scriptscriptstyle{
\eqref{takeuchiyippieh}
}}{=}}
&
        \big(1_U \otimes_A (p_{[1]} u_{(2)} \bract \gve(v_{(2)})) q \big)
\otimes_{\mkern 1mu U}
\big(u_{(1)}v_{(1)} \otimes_A
p_{[0]}\big)
       \\
       &
  {\overset{\scriptscriptstyle{
\eqref{Tch2}, \eqref{Tch9}
}}{=}}
&
  \big(
  (p_{[1]} u_{(2)})_{\smap (1)} t \gve(v_{(2)})) (p_{[1]} u_{(2)})_{\smam  }
 \otimes_A (p_{[1]} u_{(2)})_{\smap (2)}  q \big)
\otimes_{\mkern 1mu U}
\big(u_{(1)}v_{(1)} \otimes_A
p_{[0]}\big)
       \\
        &
  {\overset{\scriptscriptstyle{
\eqref{mist}, \eqref{Tch7}
}}{=}}
&
  (p_{[1]} u_{(2)})_{\smap } \big(v_{(2)\smap } 
 \otimes_A q \big) v_{(2)\smam  }  (p_{[1]} u_{(2)})_{\smam  }
\otimes_{\mkern 1mu U}
\big(u_{(1)}v_{(1)} \otimes_A
p_{[0]}\big)
       \\
        &
  {\overset{\scriptscriptstyle{
\eqref{mist}
}}{=}}
&
  (p_{[1]} u_{(2)})_{\smap } \big(v_{(2)\smap } 
 \otimes_A q \big) v_{(2)\smam  }  
\otimes_{\mkern 1mu U}
(p_{[1]} u_{(2)})_{\smam  } \big(u_{(1)} \otimes_A
p_{[0]}\big)v_{(1)}
%       \\
%
\end{eqnarray*}
  \begin{eqnarray*}
       &
  {\overset{\scriptscriptstyle{
\eqref{mampf3}
}}{=}}
&
  (p_{[1]} u_{(2)})_{\smap } \big(v_{\smap } 
 \otimes_A q \big) v_{\smam  +}  
\otimes_{\mkern 1mu U}
(p_{[1]} u_{(2)})_{\smam  } \big(u_{(1)} \otimes_A
p_{[0]}\big)v_{\smam  -}
       \\
        &
         {\overset{\scriptscriptstyle{
               \eqref{mist}
}}{=}}
&
  (u \otimes_A p)_{[1]\smap }
  (v \otimes_A q)_{\smap } (v \otimes_A q)_{\smam  +}  
\otimes_{\mkern 1mu U}
(u \otimes_A p)_{[1]\smam  } (u \otimes_A
p)_{[0]} (v \otimes_A q)_{\smam  -},  
   \end{eqnarray*}
  where we used the fact that in this case $U \otimes_A Q$ can be seen as a (right) Hopf-Galois comodule and $U \otimes_A P$ as a left one, in the sense of \S\ref{lostbuddha}. Hence, the braiding $\tau$ and the YD braiding $\gs$ are intertwined by $\zeta$ (which even saves us from explicitly proving that $\tau$ is, in fact, a braiding) and therefore the monoidal equivalence becomes a braided monoidal equivalence.
  \end{proof}

\begin{rem}
  Note that whenever $m \in  \coM $ and $ n \in N^{\mkern 1mu \mathrm{co} \, U} $ are left resp.\ right coinvariant elements, by  Eq.~\eqref{schnief} (and its left analogue) the braiding $\tau_{M,N}$ reduces to the tensor flip. Furthermore, combining Proposition \ref{resistente} with Theorem \ref{buddha}, we can define a braiding
  $$
\tau'_{M,N} \colon M \bx_U N \to N \bx_U M,
$$
on the monoidal category $\big(\tetra, \bx_U\big)$ by $\tau'_{M,N} := \xi_{N,M} \circ \tau_{M,N} \circ \xi^{-1}_{M,N}$, using \eqref{vipitenojoghurt} and~\eqref{hopfbraid}, such that \eqref{cuffie} becomes an equivalence of braided monoidal categories.
  \end{rem}

\begin{rem}
  Under the assumptions of Theorem \ref{buddha}, according to \cite[Lem.~1.1]{CheGavKow:DFAB}, there is a (strict) monoidal equivalence between $\ydu$ and $\yd$ induced by the identity on objects, that is to say, between the left and the right weak centre of the category ${}_U \cM$ (which are themselves monoidal categories). This equivalence can be easily shown to be compatible with the respective braidings $\gs$ on $\ydu$ from \eqref{ydbraid} and $\gs' \colon p \otimes_A q \mapsto p_{(-1)}q \otimes_A p_{(0)}$ on $\yd$; more precisely, in this case $\gs' = \gs^{-1}$ as seen by inserting the left coaction $\gl_P(p) = p_- \otimes_A p_+$ from \eqref{fabercastell3} into $\gs'$ given above.
  One therefore equally obtains an equivalence
  $$
\big(\tetra, \tau) \simeq \big(\yd, \gs'\big)
  $$
of braided monoidal categories.
   \end{rem}

\section{Examples}
\label{exa}

So far, only a few examples of (left or right) Hopf algebroids are known that are not full Hopf algebroids, in the sense mentioned in the introduction that instead of one, one rather has {\em two} (left and right) bialgebroid structures, plus an antipode intertwining the two. The difficulty here consists in the fact that these bialgebroid structures are of mutually opposite chirality, in the sense that a left bialgebroid, as a defining property, induces a monoidal structure on its {\em left} (but {\em not} right) modules and, accordingly, the counit is a {\em left} character that turns the base algebra into the unit in this monoidal category; while for a right bialgebroid, all this is true when examining the right handed version. An antipode is then, in particular, an anti-ring and anti-coring morphism between these two categorically different bialgebroid structures.

For example, in the context of universal enveloping algebras of Lie-Rinehart algebras, as seen in \cite[Ex.~2]{KowKra:DAPIACT}, these are always left and right Hopf algebroids over a left bialgebroid, and admit an antipode in the above sense if the base algebra is equipped with a {\em flat right Lie-Rinehart connection}, as shown in \cite[Prop.~3.12]{KowPos:TCTOHA},
which, however, may not always exist \cite{KraRov:ALRAWNA};
therefore, in these specific situations,
there cannot be a right character, or right counit, and hence the aforementioned right bialgebroid structure is missing.

In the following, we will focus on another class of examples given by one of the motivating constructions for bialgebroid theory, that is, by the so-called {\em Ehresmann-Schauenburg} bialgebroid attached to Hopf-Galois extensions with respect to a Hopf algebra $H$. These are always (left or right) Hopf algebroids over a left bialgebroid, and also full Hopf if the antipode of $H$ is involutive, which, however, is only sufficient and not necessary. We conjecture that there might be cases in which the Ehresmann-Schauenburg bialgebroid does not admit an antipode.

\subsection{The Ehresmann-Schauenburg bialgebroid}
The Ehresmann-Schauenburg bialgebroid, in its original form attached to a right Hopf-Galois extension, has been introduced
in \cite[Thm.~6.3]{Schau:BONRAASTFHB}. Since in this paper we have been mainly dealing with right Hopf algebroids, let us follow an alternative construction attached to {\em left} Hopf-Galois extensions, as presented in \cite[Rem.~2.4]{KowWeb:TSEACFHSOHA}, to which we refer for details and proofs. 

Let $H$ be a Hopf algebra over a field $k$ with invertible antipode $S$, 
let $A$ be a left $H$-comodule algebra, and let $B := {}^{\co H}\! A$ denote the subalgebra of left $H$-coinvariants.
If the map
 $$
\gamma \colon  A \otimes_B A \to H \otimes A, \quad a \otimes a' \mapsto a_{(-1)} \otimes a_{(0)} a',
 $$
is bijective, one speaks of a {\em (left) Hopf-Galois extension $B \subseteq A$}, which is
called {\em faithfully flat} if $A$ is a faithfully flat right $B$-module, or, equivalently, by applying a left handed variant of \cite[Thm.~1]{Schn:PHSFAHA}, a faithfully flat left $B$-module.
The respective {\em translation map}
is defined by
\begin{equation}
	\label{berlinale3}
\gamma^{-1} \circ (H \otimes \eta_A) \colon H \to A \otimes_B A, \quad h \mapsto h^{[1]} \otimes_B h^{[2]},
\end{equation}
where $\eta_A\colon k \to A$ denotes the unit of $A$.
Observe that the image of this map lands in the $B$-bimodule centre of $A \otimes_B A$, that is to say,
\begin{equation}
  \label{immernochsokalt}
b h^{[1]} \otimes_B h^{[2]} = h^{[1]} \otimes_B h^{[2]}b
\end{equation}
for all $b\in B$ and $h\in H$.
Similarly to \cite[Rem.~3.4]{Schn:RTOHGE}, one can show that the identities
\begin{equation}
	\label{allthat2}
	\begin{array}{rcl}
	  {h^{[ 1] }}_{\! (-1)}  \otimes  {h^{[ 1] }}_{\! (0)} {h^{[ 2] }} = h \otimes 1_A,
          &&
{a_{(-1)}}^{\! [ 1] } \otimes_B   {a_{(-1)}}^{\! [ 2 ] } a_{(0)} 
		 = a \otimes_B 1_A,
		\\[1mm]
		h^{[ 1] } h^{[ 2] } = \gve(h)1_A, &&
		(hg)^{[ 1] } \otimes_B (hg)^{[ 2] }
		= h^{[ 1 ] }g^{[ 1] } \otimes_B g^{[ 2] } h^{[ 2] },
		\\[1mm]
 {h^{[ 1] }}_{\!(-1)} \otimes  {h^{[ 1] }}_{\!(0)} \otimes_B h^{[2]}
                  & = &
h_{(1)}  \otimes
 {h_{(2)}}^{\! [ 1] } \otimes_B  {h_{(2)}}^{\! [ 2] },
		\\[1mm]
	        Sh_{(2)}
                \otimes
	  {h_{(1)}}^{\! [ 1] } \otimes_B  {h_{(1)}}^{\! [ 2] }
          & = &
          {h^{[ 2] }}_{\! (-1)} \otimes h^{[ 1] } \otimes_B  {h^{[ 2 ]}}_{\! (0)} ,
	\end{array}
\end{equation}
are true for all $h, g \in H$ and $a\in A$.
Next, consider the space
$$
\mathscr{H} := {}^{\mathrm{co}H} \! (A\otimes A)
:= \{a\otimes a'\in A\otimes A \mid a_{(-1)}a'_{(-1)} \otimes  a_{(0)}\otimes a'_{(0)} = 1 \otimes a\otimes a' \}.
$$
of left $H$-coinvariants: analogously to \cite[Thm.~6.3]{Schau:BONRAASTFHB}, this can be given the structure of a left bialgebroid over $B = {}^{\co H}\! A$, with source and target maps
$$
s,t \colon B \to \mathscr{H}, \quad  s(b):=b\otimes 1, \quad t(b):=1\otimes b,
$$
ring structure
$$
	(a\otimes a')(c\otimes c') := ac\otimes c'a', \qquad 1_{\mathscr{H}} := 1_A \otimes 1_A,
$$
and finally $B$-coring structure
\begin{equation}
	\label{blauetinte2}
	\begin{split}
	  \gD_\mathscr{H} \colon \mathscr{H} &\to  \mathscr{H} \otimes_B  \mathscr{H},
          \\
	  \gve_\mathscr{H} \colon
          \mathscr{H} &\to B,
	\end{split}
        \quad
	\begin{split}
		a\otimes a'&\mapsto 
                \big (a \otimes a'_{(-1)}{}^{\! [1]} \big) \otimes_B
                \big( a'_{(-1)}{}^{\! [2]} \otimes a'_{(0)} \big),
                \\
		a\otimes a'&\mapsto aa'.
	\end{split}
\end{equation}
As shown in \cite[Rem.~2.4]{KowWeb:TSEACFHSOHA}, this version of an Ehresmann-Schauenburg left bialgebroid can be turned into a {\em right} Hopf algebroid by defining
\begin{equation}
  \label{internetgehtnich2}
(a\otimes a')_\smap \otimes_B  (a\otimes a')_\smam
:= 
                \big (a_{(-1)}{}^{\! [1]} \otimes a'\big) \otimes_B
                \big(a_{(-1)}{}^{\! [2]} \otimes a_{(0)} \big).
\end{equation}
However, if the antipode $S$ of $H$ is involutive, then
it was shown in {\em loc.~cit.} that
the Ehresmann-Schauenburg bialgebroid ${}^{\mathrm{co}H} \! (A\otimes A)$ of left coinvariants is isomorphic to the left bialgebroid $(A\otimes A)^{\mathrm{co}H}$ based on {\em right coinvariants}, which, in turn, was shown to be a {\em left} Hopf algebroid in \cite{HanMaj:HGEATHA}. Moreover, one can then prove:

\begin{prop}
  \label{hrensko}
Let $H$ be a Hopf algebra over a field with involutive antipode $S$, and let 
$B \subseteq A$
be a faithfully flat left Hopf-Galois extension. Then the Ehresmann-Schauenburg left bialgebroid $\mathscr{H}$ of left (hence also right) coinvariants is a full Hopf algebroid, with antipode given by
$$
\mathscr{S} \colon \mathscr{H} \to \mathscr{H}, \quad a \otimes a' \mapsto a' \otimes a,
$$
that is, by the tensor flip.
  \end{prop}

\begin{proof}
  That the tensor flip yields a well-defined map ${}^{\mathrm{co}H} \! (A\otimes A) \to {}^{\mathrm{co}H} \! (A\otimes A)$ in case $S^2 = \id$ can be proven, for example, by adapting \cite[Eq.~(2.22)]{KowWeb:TSEACFHSOHA} to left coinvariants, or by applying the isomorphism $\Xi \colon {}^{\mathrm{co}H} \! (A\otimes A) \to (A\otimes A)^{\mathrm{co}H}$ in Remark 2.4 of {\em op.~cit.}
  This essentially ex\-presses the fact that the tensor flip is a left $H$-comodule endomorphism on $A \otimes A$ with respect to the codiagonal coaction, and hence one concludes that it defines an antipode in the Hopf algebroid sense, by applying (the left coinvariant version of) Proposition 3.2 in \cite{DabLanZan:HAATFQPS}.  
  \end{proof}

  Let us show by two simple examples that $S^2 = \id$ is only sufficient but not necessary, and that also the (Hopf algebroid) antipode can be more complicated than a simple tensor flip. To this end, let us consider two cases in which the Ehresmann-Schauenburg bialgebroid is actually a Hopf algebra (and hence a full Hopf algebroid over the base field $k$):

\begin{example}
  Take $A:=H$, that is, the Hopf algebra $(H, \gD_H, \gve_H, S)$ itself as  left Hopf-Galois extension of the ground field $B:=k$ with left $H$-coaction given by the coproduct, in which case the translation map \eqref{berlinale3} turns out as $h^{[1]} \otimes h^{[2]} = h_+ \otimes h_- = h_{(1)} \otimes Sh_{(2)}$, see Example \ref{factis}.  Then
one has
  an isomorphism
    $$
\varphi \colon    \mathscr{H} = {}^{\co H} \! (H \otimes H) = k \bx_H (H \otimes H) \xrightarrow{\
      \simeq \ } H, \quad h \otimes g \mapsto \gve_H(h)g,
$$of $k$-modules,
with inverse $h \mapsto Sh_{(1)} \otimes h_{(2)}$.
This map is even an isomorphism of $k$-bialgebras, seen as $k$-bialgebroids: for example,
%for $h \in H$, one has
%$
%(\gve_\mathscr{H} \circ \gvf^{-1})(h) = Sh_{(1)} h_{(2)} = \gve_H(h).
%$
%Likewise,
\begin{equation*}
  \begin{split}
\big((\gvf \otimes \gvf) \circ \gD_\mathscr{H} \circ \gvf^{-1}\big)(h)
& = \gvf (Sh_{(1)} \otimes h_{(2) +}) \otimes \gvf (h_{(2) -} \otimes h_{(3) })
\\
& = \gvf (Sh_{(1)} \otimes h_{(2)}) \otimes \gvf (Sh_{(3)} \otimes h_{(4) })
= h_{(1)} \otimes h_{(2)} = \gD_H(h).
  \end{split}
\end{equation*}
  In the same spirit, one computes (or defines) the (in this case unique) antipode $\mathscr{S}$ of $\mathscr{H}$:
\begin{equation*}
  \begin{split}
    \mathscr{S}(h \otimes g) :=
(\gvf^{-1} \circ S \circ \gvf)(h \otimes g)
    & = S^2 g_{(2)} \otimes S g_{(1)} \gve(h)
    = S^2 g_{(2)} \otimes  S (h_{(1)} g_{(1)}) h_{(2)} = S^2 g \otimes h, 
  \end{split}
\end{equation*}
for $h \otimes g \in {}^{\co H} \! (H \otimes H)$. This can now be seen as a special case of a Hopf algebroid antipode, which turns $\mathscr{H}$ into a full Hopf algebroid over $k$ (that is, a Hopf algebra), regardless of whether the antipode was involutive or not. Moreover, in case $S^2 \neq \id$, this also yields an example of an antipode not given by the flip.
\end{example}

\begin{example}
More interesting is the case in which an arbitrary left $H$-comodule algebra $A$ is an extension of $B = k$.  Adapting the proof of Theorem 3.5 in \cite{Schau:HBE} to left coinvariants, one can again show that the $k$-bialgebroid $\mathscr{H}$ is actually a Hopf algebra, with coproduct and counit as in Eqs.~\eqref{blauetinte2}, and antipode (more precisely, {\em op}-antipode, or {\em skew} antipode) given by
$$
\mathscr{S} \colon \mathscr{H} \to \mathscr{H}, \quad a \otimes a' \mapsto
{a_{(-1)}}^{\![1]} a' {a_{(-1)}}^{\![2]} \otimes a_{(0)}. 
$$
We skip the verification of this claim as it is very similar to those in {\em loc.~cit.}, but still want to underline that the antipode of $H$ does not need to be involutive and that $\mathscr{S}$ is more complicated than the flip.
\end{example}

\subsection{Hopf bimodules for the Ehresmann-Schauenburg bialgebroid}

The previous considerations allow us to generalise \cite[Thm.~3.5]{Schau:BONRAASTFHB} from $B = k$ to arbitrary $B$ as well as extend the statement (6.3) in Theorem 6.14 of {\em op.~cit.}, or rather a left coinvariant variant of it. Let us very briefly introduce all categories needed: recall first from {\em op.~cit.}\ the definition of {\em relative Hopf modules} ${}^H \! \cM_A$ resp.\ ${}^{\, H}_A \! \cM$ as right resp.\ left $A$-module in the category of left $H$-comodules, where $H$ is a Hopf algebra (over a field) and $A$ is a left $H$-comodule algebra as before. 
A {\em relative Hopf bimodule} $M \in {}^{\, H}_A \! \cM^{\phantom{H}}_A$ is then an $A$-bimodule, or a left $\Ae$-module, that is an object in ${}^H \! \cM_A \cap {}^{\, H}_A \! \cM$ with respect to the same left $H$-coaction $\gl_M$. Hence,
\begin{equation}
  \label{camelion}
  \gl_M(ama') = (ama')_{(-1)} \otimes  (ama')_{(0)}
  = a_{(-1)} m_{(-1)} a'_{(-1)} \otimes  a_{(0)} m_{(0)} a'_{(0)},
\end{equation}
for all $m \in M$ and $a, a' \in M$, and the category $({}^{\, H}_A \! \cM^{\phantom{H}}_A, \otimes_A, A)$ becomes monoidal when using the left codiagonal comodule structure. Next, by a left coinvariant version of \cite[Lem.~6.7]{Schau:BONRAASTFHB}, the left $H$-comodule algebra $A$ can also be endowed with a {\em right} $\mathscr{H}$-coaction by defining
\begin{equation}
  \label{camelion2}
  \gr_A \colon A \to A \otimes_B \mathscr{H}, \quad a \mapsto  a_{[0]} \otimes_B a_{[1]} := a_{(-1)}\!^{[1]} \otimes_B \big(a_{(-1)}\!^{[2]} \otimes a_{(0)}\big). 
\end{equation}
Note that thanks to Eqs.~\eqref{allthat2}, the algebra $A$ becomes a right $\mathscr{H}^\op$-comodule algebra, that is,
\begin{equation}
  \label{almondo}
\gr(ac) = a_{[0]} c_{[0]} \otimes_B c_{[1]} a_{[1]} 
\end{equation}
for all $a, c \in A$.
Similarly to the above, one can then introduce the category ${}^{\phantom{\mathscr{H}}}_A \!\! \cM^\mathscr{H}_A$ of $A$-bimodules in the category of right $\mathscr{H}$-comodules, that is, objects $M \in {}_A \cM_A$ equipped with a right $\mathscr{H}$-coaction $\gr_M$ such that
\begin{equation}
  \label{farmaciadelcambio}
  \begin{split}
  \gr_M(amc) &= (amc)_{[0]} \otimes_B  (amc)_{[1]}
= a_{[0]} m_{[0]} c_{[0]} \otimes_B  c_{[1]} m_{[1]} a_{[1]}
\\
&=  a_{(-1)}\!^{[1]} m_{[0]} c_{(-1)}\!^{[1]}  \otimes_B  (c_{(-1)}\!^{[2]} \otimes c_{(0)}) m_{[1]} (a_{(-1)}\!^{[2]} \otimes a_{(0)})
\end{split}
  \end{equation}
for all $m \in M$ and $a, c \in A$. The flip here on $\mathscr{H}$ is needed to make this expression well-defined over $B$, which it is thanks to Eqs.~\eqref{immernochsokalt} and \eqref{ha3}. Finally, an $(H, \mathscr{H})$-bicomodule is a right $\mathscr{H}$-comodule and a left $H$-comodule such that the right $\mathscr{H}$-coaction is left $H$-colinear, which allows to define the monoidal category $\big({}^{\,H}_A \! \cM^\mathscr{H}_A, \otimes_A, A\big)$. We have now gathered all necessary information to state the main theorem of this example section:

\begin{theorem}
  \label{weich}
  Let $H$ be a Hopf algebra over a field $k$ and $A$ a faithfully flat left Hopf-Galois extension of $B$. Then
the pair
\begin{equation}
  \label{funky2}
  {}^{\co{H}} \! (-) \colon {}^{\,H}_A \! \cM^\mathscr{H}_A
\quad \raisebox{-3pt}{$\longleftrightarrows$} \quad \ydh \ \colon \!  A \otimes_B - 
  \end{equation}
of mutually adjoint functors induces an equivalence of monoidal categories.
  \end{theorem}

\begin{proof}
  Let us first prove that the coinvariant subspace ${}^{\co {H}} \! M$ of a relative Hopf bimodule $M$ in ${}^{\,H}_A \! \cM^\mathscr{H}_A$ indeed produces a left-right YD module over $\mathscr{H}$. To this end, let $a \otimes a' \in \mathscr{H}$ and compute for $m \in {}^{\co {H}} \! M$
  \begin{eqnarray*}
    &&
\big((a \otimes a')_{(2)} m\big)_{[0]}  \otimes_B  \big((a \otimes a')_{(2)} m\big)_{[1]} (a \otimes a')_{(1)}
\\
&
  {\overset{\scriptscriptstyle{
\eqref{blauetinte2}
}}{=}}
&
\pig(\big( a'_{(-1)}{}^{\! [2]} \otimes a'_{(0)} \big) m\pig)_{[0]}  \otimes_B  \pig(\big( a'_{(-1)}{}^{\! [2]} \otimes a'_{(0)} \big) m \pig)_{[1]}  \big (a \otimes a'_{(-1)}{}^{\! [1]} \big)
\\
&
  {\overset{\scriptscriptstyle{
%\eqref{blauetinte2}
}}{=}}
&
\big( a'_{(-1)}{}^{\! [2]}m a'_{(0)} \big)_{[0]}  \otimes_B  \big( a'_{(-1)}{}^{\! [2]}m a'_{(0)} \big)_{[1]}  \big (a \otimes a'_{(-1)}{}^{\! [1]} \big)
\\
&
  {\overset{\scriptscriptstyle{
\eqref{farmaciadelcambio}
}}{=}}
&
  a'_{(-1)}{}^{\! [2]} \mkern -.5mu _{[0]} m_{[0]} a'_{(0)[0]} \otimes_B
  \big(a'_{(0)[1]}  m_{[1]} a'_{(-1)}{}^{\! [2]} \mkern -.5mu _{[1]} \big)\big(a \otimes a'_{(-1)}{}^{\! [1]} \big)
\\
&
  {\overset{\scriptscriptstyle{
\eqref{camelion2}
}}{=}}
  &
    a'_{(-2)}{}^{\! [2]} \mkern -.5mu _{(-1)} \mkern -.5mu ^{[1]} m_{[0]} a'_{(-1)}{}^{\! [1]}   \otimes_B
    \big(
a'_{(-1)}{}^{\! [2]} \otimes a'_{(0)} \big) 
    m_{[1]} \big( a'_{(-2)}{}^{\! [2]} \mkern -.5mu _{(-1)} \mkern -.5mu ^{[2]} \otimes a'_{(-2)}{}^{\! [2]} \mkern -.5mu _{(0)} \big)\big(a \otimes a'_{(-2)}{}^{\! [1]} \big)
\\
&
  {\overset{\scriptscriptstyle{
\eqref{allthat2}
}}{=}}
&
   S(a'_{(-2)})^{[1]} m_{[0]} a'_{(-1)}{}^{\! [1]}   \otimes_B
    \big(
a'_{(-1)}{}^{\! [2]} \otimes a'_{(0)} \big) 
    m_{[1]} \big( S(a'_{(-2)})^{[2]}a \otimes a'_{(-3)}{}^{\! [1]} a'_{(-3)}{}^{\! [2]} \big)
\\
&
  {\overset{\scriptscriptstyle{
\eqref{allthat2}
}}{=}}
  &
     S(a'_{(-2)})^{[1]} m_{[0]} a'_{(-1)}{}^{\! [1]}   \otimes_B
    \big(
a'_{(-1)}{}^{\! [2]} \otimes a'_{(0)} \big) 
    m_{[1]} \big( S(a'_{(-2)})^{[2]}a \otimes 1 \big)
\\
&
  {\overset{\scriptscriptstyle{
\eqref{ha3}
}}{=}}
  &
       S(a'_{(-2)})^{[1]} S(a'_{(-2)})^{[2]} a m_{[0]} a'_{(-1)}{}^{\! [1]}   \otimes_B
    \big(
a'_{(-1)}{}^{\! [2]} \otimes a'_{(0)} \big) 
    m_{[1]} 
\\
&
  {\overset{\scriptscriptstyle{
\eqref{allthat2}
}}{=}}
  &
 a m_{[0]} a'_{(-1)}{}^{\! [1]}   \otimes_B
    \big(
a'_{(-1)}{}^{\! [2]} \otimes a'_{(0)} \big) 
    m_{[1]} 
\\
&
  {\overset{\scriptscriptstyle{
\eqref{blauetinte2}
}}{=}}
&
  (a \otimes a')_{(1)} m_{[0]}  \otimes_B  (a \otimes a')_{(2)} m_{[1]},
   \end{eqnarray*}
  which is the defining Eq.~\eqref{yd2} for left-right YD modules, as desired. Here, step eight is jus\-tified by the fact that
  \begin{equation}
    \label{lendemain}
 S(a'_{(-2)})^{[1]} \otimes_B  S(a'_{(-2)})^{[2]}a  \otimes 
  a'_{(-1)}{}^{\! [1]}   \otimes_B
  a'_{(-1)}{}^{\! [2]} \otimes a'_{(0)}
  \in
  A \otimes_B B \otimes A \otimes_B A \otimes A
  \end{equation}
  if $a \otimes a' \in {}^{\co H} \! (A \otimes A)$, 
  and hence the Takeuchi property \eqref{ha3} applies, using the expression $s(b) = b \otimes 1$ for the source map. The identity \eqref{lendemain}, in turn, is seen as follows:
  \begin{eqnarray*}
    &&
    \big(A \otimes_B \gl_A \otimes A \otimes_B \otimes A \otimes A \big)
    \big( S(a'_{(-2)})^{[1]} \otimes_B  S(a'_{(-2)})^{[2]}a \otimes 
  a'_{(-1)}{}^{\! [1]}   \otimes_B
  a'_{(-1)}{}^{\! [2]} \otimes a'_{(0)}\big)
\\
&
  {\overset{\scriptscriptstyle{
%\eqref{blauetinte2}
}}{=}}
&
S(a'_{(-2)})^{[1]} \otimes_B  S(a'_{(-2)})^{[2]}\mkern -.5mu _{(-1)} a_{(-1)} \otimes S(a'_{(-2)})^{[2]} \mkern -.5mu _{(0)} a_{(0)} \otimes 
  a'_{(-1)}{}^{\! [1]}   \otimes_B
  a'_{(-1)}{}^{\! [2]} \otimes a'_{(0)}
\\
  &
  {\overset{\scriptscriptstyle{
\eqref{allthat2}
}}{=}}
&
S(a'_{(-2)})^{[1]} \otimes_B  S^2(a'_{(-3)}) a_{(-1)} \otimes S(a'_{(-2)})^{[2]} a_{(0)} \otimes 
  a'_{(-1)}{}^{\! [1]}   \otimes_B
  a'_{(-1)}{}^{\! [2]} \otimes a'_{(0)}
\\
  &
  {\overset{\scriptscriptstyle{
\eqref{allthat2}
}}{=}}
&
S(a'_{(-2)})^{[1]} \otimes_B  1_H \otimes S(a'_{(-2)})^{[2]} a \otimes 
  a'_{(-1)}{}^{\! [1]}   \otimes_B
  a'_{(-1)}{}^{\! [2]} \otimes a'_{(0)},
  \end{eqnarray*}
  using that $A$ is a left $H$-comodule algebra in the first step, whereas the last step follows from
  $$
S^2(a'_{(-1)}) \otimes a \otimes a'_{(0)} = S a_{(-1)} \otimes a_{(0)} \otimes a', 
  $$
which can be easily deduced for any $a \otimes a' \in {}^{\co H} \! (A \otimes A)$. Then, from the right $B$-flatness of $A$ one obtains \eqref{lendemain}.

Vice versa, let $P \in \ydh$ be a left-right YD module over $\mathscr{H}$. Then $A \otimes_B P$ becomes an object in ${}^{\,H}_A \! \cM^\mathscr{H}_A$ if equipped with the following left and right $A$-action, as well as with the following left $H$-coaction and right $\mathscr{H}$-coaction:
  \begin{equation}
\label{mist2}
\begin{array}{rcl}
  c(a \otimes_B p) &:=& ca \otimes_B p,
  \\
  (a \otimes_B p)c &:=& ac_{(-1)}{}^{\! [1]}   \otimes_B
  c_{(-1)}{}^{\! [2]} p  c_{(0)},
  \\
  \gl(a \otimes_B p) &:=& a_{(-1)} \otimes (a_{(0)} \otimes_B p),
  \\
  \gr(a \otimes_B p) &:=& (a_{[0]} \otimes_B p_{[0]}) \otimes_B p_{[1]} a_{[1]},
  \end{array}
  \end{equation}
  for all $p \in P$ and $a, c \in A$, where in the last line the right $\mathscr{H}$-coaction \eqref{camelion2} on $A$ is meant.
Of the numerous necessary verifications to prove that one obtains a relative Hopf bimodule in ${}^{\,H}_A \! \cM^\mathscr{H}_A$ indeed, let us only check the compatibility of the right $A$-action with the right $\mathscr{H}$-coaction, and leave the (slightly easier) rest to the reader. Indeed, for $a, c \in A$ and $p \in P$, we see that
\begin{eqnarray*}
  &&
  \rho\big((a \otimes_B p)c\big)
\\
  & = & \big((a \otimes_B p)c\big)_{[0]} \otimes_B \big((a \otimes_B p)c\big)_{[1]}
\\
&
  {\overset{\scriptscriptstyle{
\eqref{mist2}
}}{=}}
&
\pig(\big( a c_{(-1)}{}^{\! [1]}\big)_{[0]} \otimes_B \big(c_{(-1)}{}^{\! [2]} p c_{(0)} \big)_{[0]}\pig)  \otimes_B \big(c_{(-1)}{}^{\! [2]} p c_{(0)} \big)_{[1]} \big( a c_{(-1)}{}^{\! [1]}\big)_{[1]}
\\
&
  {\overset{\scriptscriptstyle{
\eqref{almondo}
}}{=}}
&
\pig(a_{[0]} c_{(-1)}{}^{\! [1]} \mkern -.5mu _{[0]} \otimes_B \big(c_{(-1)}{}^{\! [2]} p c_{(0)} \big)_{[0]}\pig)  \otimes_B \big(c_{(-1)}{}^{\! [2]} p c_{(0)} \big)_{[1]} c_{(-1)}{}^{\! [1]}\mkern -.5mu _{[1]} a_{[1]}
\\
&
  {\overset{\scriptscriptstyle{
\eqref{camelion2}
}}{=}}
&
\pig(a_{[0]} c_{(-1)}{}^{\! [1]} \mkern -.5mu _{(-1)} \mkern -.5mu ^{[1]} \otimes_B \big(c_{(-1)}{}^{\! [2]} p c_{(0)} \big)_{[0]}\pig)  \otimes_B \big(c_{(-1)}{}^{\! [2]} p c_{(0)} \big)_{[1]} \big(c_{(-1)}{}^{\! [1]} \mkern -.5mu _{(-1)} \mkern -.5mu ^{[2]} \otimes c_{(-1)}{}^{\! [1]} \mkern -.5mu _{(0)} \big)  a_{[1]}
\\
&
  {\overset{\scriptscriptstyle{
\eqref{camelion2}
}}{=}}
&
\pig(a_{[0]} c_{(-2)}{}^{\! [1]} \otimes_B \big(c_{(-1)}{}^{\! [2]} p c_{(0)} \big)_{[0]}\pig)  \otimes_B \big(c_{(-1)}{}^{\! [2]} p c_{(0)} \big)_{[1]} \big(c_{(-2)}{}^{\! [2]} \otimes c_{(-1)}{}^{\! [1]} \big)  a_{[1]}.
\end{eqnarray*}
At this point it is crucial to observe that
 $$
 c_{(-1)}{}^{\! [1]} \otimes_B \big(c_{(-1)}{}^{\! [2]} \otimes c_{(0)}\big) \in A
\otimes_B \mathscr{H},
 $$
  that is, the last two tensorands are left coinvariant (of which we omit the easy verification) and therefore it makes sense to consider
  \begin{equation}
    \label{dochnoch}
    \begin{array}{rcl}
&&
      c_{(-1)}{}^{\! [1]} \otimes_B \gD_\mathscr{H} (c_{(-1)}{}^{\! [2]} \otimes c_{(0)})
\\[2mm]
 & = &
  c_{(-1)}{}^{\! [1]} \otimes_B \big(c_{(-1)}{}^{\! [2]} \otimes c_{(0)}\big)_{(1)} \otimes_B \big(c_{(-1)}{}^{\! [2]} \otimes c_{(0)}\big)_{(2)}
 \\
 &
   {\overset{\scriptscriptstyle{
\eqref{blauetinte2}
     }}{=}}
   &
   c_{(-2)}{}^{\! [1]}
  \otimes_B \big(c_{(-2)}{}^{\! [2]} \otimes c_{(-1)}{}^{\! [1]}\big)
  \otimes_B \big(c_{(-1)}{}^{\! [2]} \otimes c_{(0)}\big).
    \end{array}
    \end{equation}
 These are precisely the terms that appear in the computation above and which allow to apply the YD condition \eqref{yd2} in a moment. Using these, we can continue our calculation above:
\begin{eqnarray*}
  &&
  \rho\big((a \otimes_B p)c\big)
\\
&
  {\overset{\scriptscriptstyle{
%\eqref{camelion2}
}}{=}}
&
\pig(a_{[0]} c_{(-2)}{}^{\! [1]} \otimes_B \big(c_{(-1)}{}^{\! [2]} p c_{(0)} \big)_{[0]}\pig)  \otimes_B \big(c_{(-1)}{}^{\! [2]} p c_{(0)} \big)_{[1]} \big(c_{(-2)}{}^{\! [2]} \otimes c_{(-1)}{}^{\! [1]} \big)  a_{[1]}.
\\
&
  {\overset{\scriptscriptstyle{
\eqref{dochnoch}
}}{=}}
&
\pig(a_{[0]} c_{(-2)}{}^{\! [1]} \otimes_B \big((c_{(-1)}{}^{\! [2]} \otimes c_{(0)})_{(2)} p \big)_{[0]}\pig)  \otimes_B \big((c_{(-1)}{}^{\! [2]} \otimes c_{(0)})_{(2)}  p\big)_{[1]} \big(c_{(-1)}{}^{\! [2]} \otimes c_{(0)}\big)_{(1)}   a_{[1]}.
\\
&
  {\overset{\scriptscriptstyle{
\eqref{yd2}
}}{=}}
&
\big(a_{[0]} c_{(-2)}{}^{\! [1]} \otimes_B (c_{(-1)}{}^{\! [2]} \otimes c_{(0)})_{(1)} p_{[0]}\big)  \otimes_B (c_{(-1)}{}^{\! [2]} \otimes c_{(0)})_{(2)}  p_{[1]} a_{[1]}.
\\
&
  {\overset{\scriptscriptstyle{
\eqref{dochnoch}
}}{=}}
&
  \big(a_{[0]} c_{(-2)}{}^{\! [1]} \otimes_B c_{(-2)}{}^{\! [2]} p_{[0]} c_{(-1)}{}^{\! [1]}
  \big)  \otimes_B (c_{(-1)}{}^{\! [2]} \otimes c_{(0)})  p_{[1]} a_{[1]}.
\\
&
  {\overset{\scriptscriptstyle{
\eqref{allthat2}
}}{=}}
&
  \big(a_{[0]} c_{(-1)}{}^{\! [1]} \mkern -.5mu _{(-1)} \mkern -.5mu ^{[1]} \otimes_B
  c_{(-1)}{}^{\! [1]} \mkern -.5mu _{(-1)} \mkern -.5mu ^{[2]}  p_{[0]} c_{(-1)}{}^{\! [1]} \mkern -.5mu _{(0)} 
  \big)  \otimes_B (c_{(-1)}{}^{\! [2]} \otimes c_{(0)})  p_{[1]} a_{[1]}.
\\
&
  {\overset{\scriptscriptstyle{
\eqref{mist2}, \eqref{camelion2}
}}{=}}
&
  (a_{[0]} \otimes_B
  p_{[0]})c_{[0]}   \otimes_B c_{[1]} p_{[1]} a_{[1]},
\end{eqnarray*}
which is Eq.~\eqref{farmaciadelcambio} for the right $A$-action, as desired. Hence, $A \otimes_B P$ is a relative Hopf module in $\cM^\mathscr{H}_A$, and likewise (but quicker) for all remaining relative Hopf module possibilities.
\end{proof}

\begin{cor}
  \label{weich2}
  Let $H$ be a Hopf algebra over a field $k$ and $B \subseteq A$ a faithfully flat left Hopf-Galois extension. Then one has an equivalence 
  $$
\ydh \simeq \hyd
$$
of monoidal categories.
  \end{cor}

\begin{proof}
 Both categories are monoidally equivalent to the category
  of relative Hopf bimodules. The first equivalence $\ydh \simeq  {}^{\,H}_A \! \cM^\mathscr{H}_A$
has just been proven in Theorem \ref{weich} above, the second equivalence $\hyd \simeq  {}^{\,H}_A \! \cM^\mathscr{H}_A$ is a left coinvariant adaption of Theorem 6.14 in \cite{Schau:BONRAASTFHB}. 
  \end{proof}

By combining Theorem \ref{weich} with part (i) of Theorem \eqref{spielzeit}, we can conclude by stating:

\begin{cor}
    \label{weich3}
  Let $H$ be a Hopf algebra over a field $k$ and $B \subseteq A$ a faithfully flat left Hopf-Galois extension. Then
  one has an equivalence
  $$
\tetrah \simeq  {}^{\,H}_A \! \cM^\mathscr{H}_A \simeq \hyd
$$
of monoidal categories.
  \end{cor}

In view of Proposition \ref{endlichwaermerjetzt?}, one can now also enunciate that the category of Hopf bimodules over an Ehresmann-Schauenburg bialgebroid is monoidally equivalent to the respective monoidal centres of the categories of left modules or right comodules (of both $H$ or $\mathscr{H}$). 

\appendix

\section{Bialgebroids, Hopf algebroids, and their (co)modules}

Let us list some facts on bialgebroids and Hopf algebroids needed explicitly in the main text. With the possible exception of the relations \eqref{mampf1}--\eqref{mampf3}, all of this is standard material.

\addtocontents{toc}{\SkipTocEntry}
\subsection{Left bialgebroids}
A \textit{left bialgebroid} \cite{Tak:GOAOAA} is a sextuple $(U ,A, s,t,\Delta,\varepsilon)$, where
\begin{enumerate}
\compactlist{99}
\item
$(U ,s,t)$ is an $\Ae$-ring such that $U$ has two commuting $\Ae$-module structures
$$
a\lact u\ract a'
:=s(a)t(a')u,\qquad\qquad
a\blact u\bract  a'
:=ut(a)s(a'); 
$$
\item
$(U,\Delta,\varepsilon)$ is an $A$-coring with respect to the left and right $A$-actions $\lact$ and $\ract$ and coproduct $\Delta\colon U \to U _\ract\times_A\due {U } \lact {}$, where
  \begin{equation}
  \label{takeuchiyippieh}
  U _\ract\times_A\due {U } \lact {}=
  \big\{\Sum\nolimits_i u'_i\otimes_Au_i\in U _\ract\otimes_A\due {U } \lact {} \mid 
a\blact u'_i\otimes_A u_i=u'_i\otimes_Au_i\bract a,\ \forall a\in A \big\}
\end{equation}
  is called the {\em Takeuchi-Sweedler} subspace (or product, which is an $\Ae$-ring as well).
  %and in the quotient $U _\ract \otimes_A \due U \lact {}$ we identify $u\ract %a\otimes u'$ with $u\otimes a\lact u'$.
In particular, the coproduct $\gD$ is $\Ae$-bilinear in the sense of 
$$
\Delta(a\lact u\ract a')
=a\lact u_{(1)}\otimes_A u_{(2)}\ract a',\qquad\qquad
\Delta(a\blact u\bract a')
=u_{(1)}\bract  a'\otimes_Aa\blact u_{(2)},
$$
where we wrote $\Delta(u) = u_{(1)}\otimes_Au_{(2)}$;
\item
  the counit  $\varepsilon\colon U \to A$ is an $\Ae$-linear map
that satisfies
\begin{equation}
  \label{counityippieh}
\varepsilon(a\lact u\ract a')
=a\varepsilon(u)a',\qquad\qquad
\varepsilon(uu') =
\varepsilon(u\bract  \varepsilon(u'))
=\varepsilon(\varepsilon(u')\blact u),
\end{equation}
that is, it is not a ring morphism in general. 
\end{enumerate}
We mostly refer to the datum of a left bialgebroid by simply writing
$(U, A)$ or only $U$.

\addtocontents{toc}{\SkipTocEntry}
\subsection{Modules over left bialgebroids}
\label{casadante}
A ({\em left} or {\em right}) {\em module} over a left bialgebroid $(U, A)$ is a module over the underlying ring structure of the total space $U$. While the category ${}_U \cM$ of left modules is monoidal (with monoidal product $\otimes_A$, diagonal $U$-action, and unit the base algebra $A$), the category $\cM_U$ of right modules in general is not, which is one of the reasons why the definition of (left-right or right-right) Hopf modules in the main text is more intricate. Such lack of structure notwithstanding, for both categories there are forgetful functors ${}_U \cM \to {}_\Ae \cM$ resp.\ $\cM_U \to {}_\Ae \cM$, with respect to which we denote
\begin{equation}
  \label{soedermalm}
a \lact m \ract a' := s(a)t(a')m, \qquad a \blact n \bract a' := nt(a)s(a')
  \end{equation}
for $a,a'\in A$, $m\in M$, and $n\in N$, where $M$ resp.\ $N$ is a left resp.\ right $U $-module.

\addtocontents{toc}{\SkipTocEntry}
\subsection{Comodules over left bialgebroids}
\label{schleifmaschine}
A \textit{left comodule} over a left bialgebroid $(U,A)$ consists of a pair $(M,\lambda_M)$, where
$M\in{}_A\mathcal{M}$ is, a priori, a left $A$-module (with action denoted by a central dot $\cdot$), and where the map $\lambda_M\colon M\to U \otimes_AM, \ m \mapsto m_{(-1)}\otimes_Am_{(0)}$ is a
left $A$-linear left coaction over the underlying coring structure of the total space $U$.
This datum implies the existence of a right $A$-action $m \cdot a:=\varepsilon(m_{(-1)}\bract a) \cdot m_{(0)}$, which we shall call the {\em induced} right $A$-action on $M$, and with respect to which the coaction becomes right $A$-linear as well. Hence, 
\begin{equation}
	\label{hak}
	\lambda_M(a \cdot m \cdot a')
	=a\lact m_{(-1)}\bract a'\otimes_Am_{(0)},\qquad
	m_{(-1)}\otimes_Am_{(0)} \cdot a  =  a\blact m_{(-1)}\otimes_Am_{(0)}.
\end{equation}
Here, the identity on the right hand side expresses the fact that the left coaction
$\lambda_M$ {\em corestricts} to the Takeuchi-Sweedler subspace
\begin{equation*}
U_\ract \times_A M = \big\{ \textstyle\sum_i u^i \otimes m^i \in U_\ract \otimes_A M \mid  \sum_i a \blact u^i \otimes m^i = \sum_i u^i \otimes m^i \cdot a, \ \forall a \in A  \big\}.
\end{equation*}
The category of left $U $-comodules shall be denoted by ${}^U \mkern -5mu \cM$. Again, one has a forgetful functor ${}^U \mkern -5mu \cM \to {}_\Ae \cM$, by construction.

{\em Right} $U$-comodules are defined along the same lines: a pair $(N,\rho_N)$ of a right $A$-module $N \in \cM_A$, with action denoted just by juxtaposition,
and a
right $A$-linear $U$-coaction $\rho_N\colon N\to N \otimes_A \due U \lact {}$, which, in turn, {\em induces} a left $A$-action with respect to which the coaction becomes linear as well and, in particular, corestricts to the Takeuchi-Sweedler subspace $N \times_A U$, that~is,
\begin{equation}
  \label{ha3}
\rho_M(a n a' )
= n_{[0]} \otimes_A a \blact n_{[1]} \ract a',
\qquad
an_{[0]}\otimes_A n_{[1]}  = n_{[0]}\otimes_A n_{[1]} \bract a,
\end{equation}
where $an := n_{[0]} \gve(n_{[1]}\bract a)$ denotes the aforementioned induced left $A$-action, and where we employed the Sweedler notation $\rho_M(n)=n_{[0]}\otimes_A n_{[1]}$ with square brackets for better distinction. We shall denote by $\comodu$
the category of right $U$-comodules. As for left comodules, there is a forgetful functor $\cM^U \to {}_\Ae \cM$.
 
%% \addtocontents{toc}{\SkipTocEntry}
%% \subsection{Yetter-Drinfel'd modules}
%% A {\em left-right Yetter Drinfel'd} module $P$ over a left bialgebroid $(U,A)$ is
%% simultaneously a left $U$-module (with action denoted by juxtaposition) and a right $U$-comodule with coaction $\gr_P \colon P \to P \otimes_A \due U \lact {}, \ p \mapsto p_{[0]} \otimes_A p_{[1]}$ such that the two forgetful functors $\umod \to {}_\Ae \cM$ and $\comodu \to {}_\Ae \cM$ induce the same $A$-bimodule structure on $P$,
%% \begin{equation}
%%   \label{ydforget2}
%% a \lact p \ract a' = apa'
%% \end{equation}
%% for $a, a' \in A$ and $p \in P$,
%% and such that between $U$-action and $U$-coaction 
%%  the compatibility
%% \begin{equation}
%%   \label{yd2}
%% u_{(1)} p_{[0]} \otimes_A u_{(2)} p_{[1]} = (u_{(2)} p)_{[0]}  \otimes_A  (u_{(2)} p)_{[1]}u_{(1)}
%% \end{equation}
%% holds for all $u \in U$ and all $p \in P$. The corresponding category $\ydu$ is braided  monoidal with respect to $\otimes_A$ and with braiding mentioned in \eqref{ydbraid}, unit object given by the base algebra $A$; it can be shown that it is (monoidally) equivalent to the (right weak) monoidal centre of the category $\umod$.
%% Observe that
%% one can still define {\em left-left} YD modules, but in contrast to the case of bialgebras,
%% there are {\em no} right-left or right-right YD modules over left bialgebroids.

\addtocontents{toc}{\SkipTocEntry}
\subsection{Left and right Hopf algebroids over left bialgebroids}
\label{sokc}

Generalising Hopf algebras ({\em i.e.}, bialgebras with an antipode) to noncommutative base rings is a challenging task.
If one wants to avoid the abundance of structure maps that accompany the notion of a {\em full} Hopf algebroid as in \cite{BoeSzl:HAWBAAIAD}, that is, {\em two} bialgebroid structures (meaning two coproducts, two counits, eight $A$-actions on the total space, {\em etc.}) and an antipode map as sort of intertwiner between all of this, one renounces on the idea of an antipode and rather requires a certain Hopf-Galois map to be invertible {\cite{Schau:DADOQGHA}}: this leads to a more general concept than that of full Hopf algebroids.
More precisely, if $(U, A)$ is a left bialgebroid, consider the maps
\begin{equation}
  \label{nochmehrRegen}
\begin{array}{rclrcl}
\ga_\ell  \colon  \due U \blact {} \otimes_{\Aop} U_\ract &\to& U_\ract  \otimes_A  \due U \lact,
 & u \otimes_\Aop v  &\mapsto&  u_{(1)} \otimes_A u_{(2)}  v,
 \\
\ga_r  \colon  U_{\!\bract}  \otimes_A \! \due U \lact {}  &\to& U_{\!\ract}  \otimes_A  \due U \lact,
&  u \otimes_A v  &\mapsto&  u_{(1)}  v \otimes_A u_{(2)},
\end{array}
\end{equation}
of left $U$-modules.
Then a left bialgebroid $(U,A)$ is called
a
{\em left Hopf algebroid} or simply {\em left Hopf} if $ \alpha_\ell $ is invertible and
{\em right Hopf algebroid} or {\em right Hopf}
%if $ \alpha_r $ is invertible.
if this is the case for $\ga_r$.
Adopting two different kind of Sweedler notations
\begin{equation*}
  \begin{array}{rcl}
 u_+ \otimes_\Aop u_-  & \coloneqq &  \alpha_\ell^{-1}(u \otimes_A 1)
 \\
   u_{\smap } \otimes_A u_{\smam  }  & \coloneqq &  \alpha_r^{-1}(1 \otimes_A u),
\end{array}
  \end{equation*}
with, as usual, summation understood, one proves that
for a left Hopf algebroid
\begin{eqnarray}
\label{Sch1}
u_+ \otimes_\Aop  u_- & \in
& U \times_\Aop U,  \\
\label{Sch2}
u_{+(1)} \otimes_A u_{+(2)} u_- &=& u \otimes_A 1 \quad \in U_{\!\ract} \! \otimes_A \! {}_\lact U,  \\
\label{Sch3}
u_{(1)+} \otimes_\Aop u_{(1)-} u_{(2)}  &=& u \otimes_\Aop  1 \quad \in  {}_\blact U \! \otimes_\Aop \! U_\ract,  \\
\label{Sch4}
u_{+(1)} \otimes_A u_{+(2)} \otimes_\Aop  u_{-} &=& u_{(1)} \otimes_A u_{(2)+} \otimes_\Aop u_{(2)-},  \\
\label{Sch5}
u_+ \otimes_\Aop  u_{-(1)} \otimes_A u_{-(2)} &=&
u_{++} \otimes_\Aop u_- \otimes_A u_{+-},
\\
\label{Sch6}
(uv)_+ \otimes_\Aop  (uv)_- &=& u_+v_+ \otimes_\Aop v_-u_-,
\\
\label{Sch7}
u_+u_- &=& s (\varepsilon (u)),
\\
\label{Sch8}
\varepsilon(u_-) \blact u_+  &=& u,
\\
\label{Sch9}
(s (a) t (a'))_+ \otimes_\Aop  (s (a) t (a') )_-
&=& s (a) \otimes_\Aop s (a')
\end{eqnarray}
are true \cite{Schau:DADOQGHA},
where in  \eqref{Sch1}  we mean the Takeuchi-Sweedler product
\begin{equation*}
\label{petrarca}
   U \! \times_\Aop \! U    \coloneqq
   \big\{ {\textstyle \sum_i} u_i \otimes v_i \in {}_\blact U  \otimes_\Aop  U_{\!\ract} \mid {\textstyle \sum_i} u_i \ract a \otimes v_i = {\textstyle \sum_i} u_i \otimes a \blact v_i, \ \forall a \in A \big\},
\end{equation*}
and if the left bialgebroid $(U,A)$ is right Hopf, in the same spirit one verifies
\begin{eqnarray}
\label{Tch1}
u_{\smap } \otimes_A  u_{\smam  } & \in
& U \times_A U,  \\
\label{Tch2}
u_{\smap (1)} u_{\smam  } \otimes_A u_{\smap (2)}  &=& 1 \otimes_A u \quad \in U_{\!\ract} \! \otimes_A \! {}_\lact U,  \\
\label{Tch3}
u_{(2)\smap } \otimes_A u_{(2)\smam  }u_{(1)}  &=& u \otimes_A 1 \quad \in U_{\!\bract} \!
\otimes_A \! \due U \lact {},  \\
\label{Tch4}
u_{\smap (1)} \otimes_A u_{\smam  } \otimes_A u_{\smap (2)} &=& u_{(1)\smap } \otimes_A
u_{(1)\smam  } \otimes_A  u_{(2)},  \\
\label{Tch5}
u_{\smap \smap } \otimes_A  u_{\smap \smam  } \otimes_A u_{\smam  } &=&
u_{\smap } \otimes_A u_{\smam  (1)} \otimes_A u_{\smam  (2)},  \\
\label{Tch6}
(uv)_{\smap } \otimes_A (uv)_{\smam  } &=& u_{\smap }v_{\smap }
\otimes_A v_{\smam  }u_{\smam  },  \\
\label{Tch7}
u_{\smap }u_{\smam  } &=& t (\varepsilon (u)),  \\
\label{Tch8}
u_{\smap } \bract \varepsilon(u_{\smam  })  &=&  u,  \\
\label{Tch9}
(s (a) t (a'))_{\smap } \otimes_A (s (a) t (a') )_{\smam  }
&=& t(a') \otimes_A t(a),
\end{eqnarray}
see \cite[Prop.~4.2]{BoeSzl:HAWBAAIAD},
where in  \eqref{Tch1} we denoted
\begin{equation*}  \label{petrarca2}
   U \times_A U    \coloneqq
   \big\{ {\textstyle \sum_i} u_i \otimes  v_i \in U_{\!\bract}  \otimes_A \!  \due U \lact {} \mid {\textstyle \sum_i} a \lact u_i \otimes v_i = {\textstyle \sum_i} u_i \otimes v_i \bract a,  \ \forall a \in A  \big\}.
\end{equation*}

If the left bialgebroid $(U,A)$ is simultaneously left and right Hopf, the compatibility between the two (inverses of the) Hopf-Galois maps comes out as:
\begin{eqnarray}
\label{mampf1}
u_{+\smap } \otimes_\Aop u_{-} \otimes_A u_{+\smam  } &=& u_{\smap +} \otimes_\Aop u_{\smap -} \otimes_A u_{\smam  }, \\
\label{mampf2}
u_+ \otimes_\Aop u_{-\smap } \otimes_A u_{-\smam  } &=& u_{(1)+} \otimes_\Aop u_{(1)-} \otimes_A u_{(2)}, \\
\label{mampf3}
u_{\smap} \otimes_A u_{{\smam}+} \otimes_\Aop u_{\smam  -} &=& u_{(2)\smap } \otimes_A u_{(2)\smam  } \otimes_\Aop u_{(1)},
\end{eqnarray}
see \cite[Lem.~2.3.4]{CheGavKow:DFOLHA}.

%% A simultaneous left and right Hopf structure on a left bialgebroid still does not imply the existence of an antipode required in the definition of a full Hopf algebroid.
%

\begin{example}
  \label{factis}
In case $(U, A) = (H, k)$ is actually a Hopf algebra over a field $k$,
the invertibility of $\ga_\ell$ guarantees the existence of the antipode $S$, while
the invertibility of $\ga_r$ guarantees the existence of the inverse  (or {\em op-antipode} for that matter) $S^{-1}$ of the antipode.
More precisely,
%for any $h \in H$ we have
\begin{equation*}
  \label{sesam}
  \begin{array}{rcl}
     h_+ \otimes_k h_- &=& h_{(1)} \otimes_k S(h_{(2)})
     \\[2pt]
    h_{\smap } \otimes_k h_{\smam  } &=& h_{(2)} \otimes_k S^{-1}(h_{(1)}).
  \end{array}
  \end{equation*}
in
this case, for any $h \in H$.
\end{example}

\end{document}